\documentclass[12pt]{amsart}

\usepackage{amsmath,mathtools,mathrsfs,amssymb,bm}
\usepackage{graphicx,subfigure}
\usepackage{algorithmicx,algorithm}
\usepackage[margin=1.0in]{geometry}
\usepackage{multirow}
\usepackage{enumitem}
\usepackage{xcolor}
\usepackage{titlesec}
\usepackage{stmaryrd}
\usepackage{booktabs}
\usepackage{bm}

\theoremstyle{definition}
\newtheorem{theorem}{Theorem}[section]
\newtheorem{lemma}[theorem]{Lemma}

\numberwithin{equation}{section}
\numberwithin{figure}{section}
\numberwithin{table}{section}
\newtheorem{assumption}[theorem]{Assumption}
\newtheorem{Example}[theorem]{Example}
\newtheorem{Algorithm}[theorem]{Algorithm}

\graphicspath{{./pic}}

\setenumerate[1]{leftmargin=2.5em,itemsep=0pt,partopsep=0pt,parsep=\parskip,topsep=5pt}
\setitemize[1]{leftmargin=2.5em,itemsep=2pt,partopsep=0pt,parsep=\parskip,topsep=5pt}
\setdescription[1]{leftmargin=2.5em,itemsep=0pt,partopsep=0pt,parsep=\parskip,topsep=5pt}

\titleformat{\section}[block]{\bfseries\filcenter}{\thesection}{5pt}{}
\titleformat{\subsection}[hang]{\bfseries}{\thesubsection}{5pt}{}
\begin{document}
\include{macros}

\pagestyle{plain}
\setlength\lineskiplimit{4pt}
\setlength\lineskip{4pt}
\setlength\parskip{2pt}

\begingroup

\title{A novel shape optimization approach for source identification in elliptic equations}
\author{Wei Gong and Ziyi Zhang}
\thanks{$^*$School of Mathematical Sciences, University of  Chinese Academy of Sciences \& LSEC, Institute of Computational Mathematics, Academy of Mathematics and Systems Science, Chinese Academy of Sciences, Beijing 100190, China. Email: wgong@lsec.cc.ac.cn; zhangziyi21@mails.ucas.ac.cn. The authors acknowledge the support from the National Key Research and
Development Program of China (project no. 2022YFA1004402) and
the National Natural Science Foundation of China (project no. 12071468).}

\begin{abstract}
In this paper, we propose a novel shape optimization approach for the source identification of elliptic equations. This identification problem arises from two application backgrounds: actuator placement in PDE-constrained optimal controls and the regularized least-squares formulation of source identifications. The optimization problem seeks both the source strength and its support. By eliminating the variable associated with the source strength, we reduce the problem to a shape optimization problem for a coupled elliptic system, known as the first-order optimality system. As a model problem, we derive the shape derivative for the regularized least-squares formulation of the inverse source problem and propose a gradient descent shape optimization algorithm, implemented using the level-set method. Several numerical experiments are presented to demonstrate the efficiency of our proposed algorithms.  
\end{abstract}

\date{Modified: \today}

\maketitle


\section{Introduction}
\setcounter{equation}{0}
In this paper we are interested in the source identification of the following elliptic equations
\begin{equation}\label{source_id}
\left\{
\begin{aligned}
-\Delta u+u&=q(x){\chi}_{\omega} \quad &\mbox{in}\ \Omega,\\
\partial_nu&=u_n \quad &\mbox{on}\ \Gamma,\\
u&=u_d \quad &\mbox{on}\ \Gamma,
\end{aligned}
\right.
\end{equation}
where $\Omega\subset\mathbb{R}^d$ ($d=2,3$) is an open bounded domain with Lipschitz boundary $\Gamma:=\partial\Omega$, $u_n,u_d\in L^2(\Gamma)$ are given functions, $\omega\subset\Omega$ is the support of the source with $q(x)$ its strength, which are both unknown. The aim of this paper is to recover both the support $\omega$ and the strength $q$ from the Cauchy data. 
These problems arise in many important applications, among which we take the actuator placement in PDE-constrained controls and the inverse source problem as examples. 

Optimal actuator placement is a pervasive challenge in PDE-constrained control problems. There are primarily three types of control device distributions in applications: the distributed type, where actuators operate within a subdomain $\omega \subset \Omega$; the boundary type, where actuators are placed on parts of the boundary $\partial \Omega$; and pointwise sources. The selection of actuator placement algorithms is highly dependent on the specific application context. 

One approach for designing open-loop control actuators is to determine both the actuator locations and the corresponding optimal controls simultaneously. In \cite{Stadler2009}, the authors proposed using the $L^1$-sparse control method to optimize the placement of control devices. Here, the support of the sparse control identifies the device placement, while the control profiles determine the strength. An alternative method employs measure-valued sparse controls, as discussed in \cite{ClasonKunisch12}. Another principle for actuator design involves optimizing stabilization or controllability through the appropriate choice of controllers. This method is detailed in the references cited in \cite{Morris2018}, where actuators and sensors are positioned to maximize the controllability and observability of the least controllable and observable state, as outlined in \cite{KubruslyMalebranche1985}. In \cite{MunchPeriago2011}, the authors investigated a homogenized version of the optimal placement of internal null controllers with minimal $L^2$-norm for the 1D heat equation problem, considering fixed initial data and imposing a measure constraint on the support without any topological assumptions regarding the number of connected components. This approach was extended in \cite{MunchPeriago2013} to study the numerical approximation of bang-bang controls for the heat equation. In \cite{PrivatTrelatZuazua2017}, the moment method was utilized for actuator design in parabolic distributed parameter systems. The authors formulated a spectral optimal design problem aimed at maximizing a criterion corresponding to an average over random initial data of the largest $L^2$-energy of controllers. \cite{KaliseKunischSturm2018} considered two types of actuator design problems for parabolic equations, based on shape and topology calculus. The first problem determines the best actuators for a given initial condition, while the second identifies optimal actuators for all initial conditions not exceeding a chosen norm. The authors developed a numerical algorithm for optimal actuator design using shape and topological sensitivities of these cost functionals along with a level-set method. For further results on stabilization or controllability-based approaches, see \cite{VaidyaRajaram2012}.

On the other hand, source identifications of the form (\ref{source_id}) in the context of inverse problems have also been extensively studied. A significant application is bioluminescence tomography (BLT) (cf. \cite{ChengGongHan, ChengGongHanZheng, DingGongLiuLo, HanCongWang}). To the best of our knowledge, most existing results focus on the identification of point sources (\cite{Andrle, CasasKunisch2019, BadiaHamdi2005, MongeZuazua}) or distributed sources with known supports, where the problem reduces to inverting for the source strength (\cite{ChengGongHanZheng, HanCongWang}). For source terms in the form of a characteristic function, i.e., $\chi_\omega$ for some $\omega \subset \Omega$, see \cite{HettlichRundell1996} for uniqueness and recovery algorithms. In \cite{MongeZuazua}, the authors address the recovery of sparse initial conditions in diffusion-advection equations by a final state distribution, constituted by a finite combination of Kronecker deltas, identifying their locations and masses. In \cite{DingGongLiuLo}, the authors demonstrate the theoretical uniqueness of the BLT problem where the light sources are shaped as $C^2$ domains, polyhedral, or corona-shaped.

PDE-constrained shape and topology optimization is a classical research topic that has seen significant development over the past few decades. For a comprehensive review of recent advancements, we refer to \cite{Allaire, Sokolowski}. The shape optimization approach has been employed to recover obstacles in inverse scattering problems (cf. \cite{Hettlich,Kirsch}) and to identify the support of the source in both steady and unsteady diffusion and subdiffusion problems (cf. \cite{HettlichRundell2001, HettlichRundell1996, HuZhu}). Additionally, its application in electrical impedance tomography is documented in \cite{Afraites, HintermullerLaurain}. However, to the best of our knowledge, there are no existing results on its application to source identification problems of the form (\ref{source_id}).

The aim of this paper is to simultaneously identify the source support and intensity in the form (\ref{source_id}). Using Cauchy data, we formulate an optimization problem employing the least-squares method and Tikhonov regularization, which seeks both the source strength and its support. When the support $\omega$ is known, the optimization problem simplifies to a classical inverse source problem. In this case, the first-order optimality condition links the unknown source with the adjoint, allowing us to eliminate the variable associated with the source strength. This reduction transforms the optimization problem into a shape optimization problem for a coupled elliptic system, referred to as the first-order optimality system. We remark that our method is inspired by \cite{ANSS}, where the distributed and shape controls of  Stokes equations were considered. As a model problem, we derive the shape derivative for the regularized least-squares formulation of the inverse source problem and propose a gradient descent shape optimization algorithm, implemented via the level-set method. Several numerical experiments are presented to demonstrate the efficiency of our proposed algorithms, covering scenarios with smooth or polygonal support, constant or varying intensity, and simply or multiply connected support.

The remainder of this paper is organized as follows: In Section 2, we present the motivation for addressing the source identification problem, using the control actuator placement and BLT problems as examples. Section 3 discusses a model inverse source problem, formulates the shape optimization problem, and proves the existence of optimal shapes. In Section 4, we demonstrate the shape differentiability of the state equation with respect to the domain and conduct shape sensitivity analysis. Section 5 is dedicated to extensive numerical experiments that illustrate the effectiveness of the proposed algorithm. Finally, in the last section, we draw conclusions and suggest potential future work. Throughout this paper, we use standard notations for Sobelev spaces and their associated norms. 



\section{Motivations for source identification}
\setcounter{equation}{0}

In this section, we present the motivation for studying inverse source problems. Among the extensive applications, we highlight those arising from actuator placement in PDE-constrained optimal control and bioluminescence tomography (BLT) problems.

\subsection{Actuator placement in optimal controls}
For a given measurable subset $\omega \subset \Omega$, representing the location of the control actuator, we can formulate the following optimal control problem: 
\begin{eqnarray}
\min\limits_{q\in U_{ad}\subset L^2(\omega)} \ J_{\omega}(u,q)\label{OCP_obj_abs}
\end{eqnarray}
subject to 
\begin{equation}\label{OCP_elliptic}
\left\{
\begin{aligned}
-\Delta u+u&=q(x){\chi}_{\omega} \quad &\mbox{in}\ \Omega,\\
\partial_nu&=u_n \quad  &\mbox{on}\ \Gamma,
\end{aligned}
\right.
\end{equation}
where $u_n\in L^2(\Gamma)$ is a given function. 

To ensure the well-posedness of the optimization problem (\ref{OCP_obj_abs})-(\ref{OCP_elliptic}) we first state some assumptions about the objective functional $J_\omega(u,q)$. 
\begin{assumption}\label{Ass:1}
	We make the following assumptions:
	\begin{enumerate}
		\item Either $J_{\omega}(u,q)\rightarrow +\infty$ as $\|q\|_{0,\omega}\rightarrow +\infty$, or $U_{ad}$ is a bounded, closed and convex subset of $L^2(\omega)$;
		
		\item $J_{\omega}(u,q)$ is sequentially weakly lower semi-continuous with respect to $u$ and $q$. 
	\end{enumerate}
\end{assumption}
A commonly used candidate for the objective functional $J_{\omega}(u,q)$ in the literature is of the tracking type:
\begin{eqnarray}
J_{\omega}(u,q):=\frac{1}{2}\|u-u_d\|_{0,\Gamma}^2+\frac{\alpha}{2}\|q\|_{0,\omega}^2,\nonumber
\end{eqnarray}
where $u_d$ is a given target state defined on $\Gamma$, and the second term with $\alpha\geq 0$ is a regularization or control cost term. If $\alpha=0$ we require that $U_{ad}$ is bounded. 

Under Assumption \ref{Ass:1} it is well-known that the above optimization problem admits a unique solution $q_{\omega}\in U_{ad}$ (cf. \cite{HPUU2009}) with associated state $u_{\omega}$, depending on $\omega$, such that
\begin{eqnarray}
J_{\omega}(u_{\omega},q_{\omega})=\min\limits_{q\in U_{ad}} \ J_{\omega}(u,q).\nonumber
\end{eqnarray}
Then we can calculate the Gateaux derivative of $J_{\omega}(u,q)$ as follows:
\begin{eqnarray}
J'_{\omega}(u_{\omega},q_{\omega})v=\partial_qJ_{\omega}(u_{\omega},q_{\omega} )v+(p_{\omega},v)_{\omega}\quad \forall v\in L^2(\omega),
\end{eqnarray}
where $p_{\omega}$ is the adjoint state solving
\begin{equation}\label{OPT_adjoint}
\left\{
\begin{aligned}
-\Delta p_{\omega}+p_{\omega}&=\partial_uJ_{\omega}(u_{\omega},q_{\omega}) \quad &\mbox{in}\ \Omega,\\
\partial_np_{\omega}&=0 \quad &\mbox{on}\ \Gamma,
\end{aligned}
\right.
\end{equation}
if the integral of $u$ in $J_{\omega}(u,q)$ is of distributed type. Therefore, we have the following first order optimality condition for the optimization problem
\begin{eqnarray}
\partial_qJ_{\omega}(u_{\omega},q_{\omega})(v-q_{\omega})+(p_{\omega},v-q_{\omega})_{\omega}\geq 0\quad \forall v\in U_{ad}.
\end{eqnarray}

However, in many cases, the control actuator location $\omega \subset \Omega$ needs to be determined, leading to the simultaneous identification of the actuator location and the optimal control. As a model problem, we consider the following optimization problem:
\begin{eqnarray}
\min\limits_{\omega\subset\Omega,\ q_\omega\in U_{ad}}\ J_{\omega}(u_{\omega},q_{\omega})\label{OCP_obj}
\end{eqnarray}
subject to (\ref{OCP_elliptic}) and the volume constraint
\begin{eqnarray}
|\omega|= \gamma_0,\quad 0<\gamma_0 <|\Omega|,\label{OCP_constraint}
\end{eqnarray}
where $\omega\subset\Omega$ denotes the location of the sensor, and $q(x)$ denotes the control function acting on $\omega$. Note that $\omega$ may not be connected. In this formulation  the sensor locations $\omega$ are the optimization variables, while the control functions $q(x)$ depend implicitly on $\omega$.


\subsection{Source identification in BLT problems}
In this subsection, we consider the identification of sources for elliptic equations, inspired by applications in bioluminescence tomography (BLT). BLT is a whole-body imaging technique that captures light produced within the body and transmitted through tissue, utilizing reporter genes that encode fluorescent or bioluminescent proteins. The objective is to reconstruct bioluminescence signals, which are internal light sources, from external measurements of the Cauchy data. For further details on applications, background, and mathematical derivations, we refer to \cite{ChengGongHan,DingGongLiuLo,HanCongWang}.

The propagation of light through a random medium is described by the radiative transfer equation, which is typically approximated by a diffusion equation. After the injection of luciferin, the bioluminescence signal varies and eventually peaks. Measurements are generally taken at peak emission, as the internal bioluminescence distribution induced by reporter genes remains relatively stable at this time, allowing us to neglect time dependence. Discarding all time-dependent terms, the stationary BLT model is given by:
\begin{equation}
\left\{
\begin{aligned}
-\nabla\cdot (D\nabla u)+\mu_0u&=q \quad &\mbox{in}\ \Omega,\\
u+2D\partial_n u&=g_1 \quad &\mbox{on}\ \Gamma,\\
D\partial_n u=g_2 \ \mbox{or} \ u&=g_1-2g_2 \quad &\mbox{on}\ \Gamma,
\end{aligned}
\right.
\end{equation}
where $g_1$ is a given function and $g_2$ is the measurement on the boundary. Our aim is to recover the source $q$ from the overdetermined elliptic system.  

In many applications, the source $q$ has compact support within $\Omega$ and takes the form 
\begin{eqnarray}
q(x)=q_0(x)\chi_{\omega},\nonumber
\end{eqnarray}
where $\omega\subset\Omega$ is the support of the source and $q_0$ is the strength, both of which are unknown. Using standard Tikhonov regularization, we can formulate the following inversion problem:
\begin{eqnarray}
\min\limits_{\omega\subset \Omega,\ q_0\in L^2(\omega)}\ J(\omega,u,q_0)=\frac{1}{2}\|u-g_1+2g_2\|_{0,\Gamma}^2+\frac{\alpha}{2}\|q_0\|^2_{0,\omega}
\end{eqnarray}
subject to 
\begin{equation}
\left\{
\begin{aligned}
-\nabla\cdot (D\nabla u)+\mu_0u&=q_0\chi_{\omega} \quad &\mbox{in}\ \Omega,\\
u+2D\partial_n u&=g_1 \quad &\mbox{on}\ \Gamma,
\end{aligned}
\right.
\end{equation}
where $\alpha>0$ is a regularization parameter. For accurate identification, we can consider multiple measurements. Thus, the model is extended to:
\begin{eqnarray}
\min\limits_{\omega\subset \Omega,\ q_0\in L^2(\omega)}\ J(\omega,u,q_0)=\frac{1}{2m}\sum\limits_{i=1}^m\|u^i-g_1^i+2g_2^i\|_{0,\Gamma}^2+\frac{\alpha}{2}\|q_0\|^2_{0,\omega}
\end{eqnarray}
subject to 
\begin{eqnarray}
\left\{
\begin{aligned}
-\nabla\cdot (D\nabla u^i)+\mu_0u^i&=q_0\chi_{\omega} \quad &\mbox{in}\ \Omega,\\
u^i+2D\partial_n u^i&=g_1^i \quad &\mbox{on}\ \Gamma.
\end{aligned}
\right.
\end{eqnarray}
We note that a similar identification problem was studied in \cite{HintermullerLaurain} where the unknown is the conductivity of the elliptic equations. 

\section{A model source identification problem}
In this section, we consider a general model for inverse source problems governed by the following elliptic equation:
 \begin{equation}\label{Poisson}
\left\{
\begin{aligned}
-\Delta u+u&= f+q(x)\chi_\omega  \quad&\mbox{in}\ \Omega,\\
u&=u_d \quad&\mbox{on}\ \Gamma,\\
\partial_nu&=u_n \quad &\mbox{on}\ \Gamma,
\end{aligned}
\right.
\end{equation}
where $\omega \subset \Omega$ is the support of the source and $q(x)$ is the intensity, both of which are unknown and need to be recovered. For given $\omega$ and boundary data $u_d$, $u_n$, the identification of the unknown source $q(x)$ is relatively straightforward. We can use the least-squares method to formulate the following optimization problem:
\begin{eqnarray}
\min\limits_{q\in L^2(\omega)}\ J_0(u,q)=\frac{1}{2}\|u-u_d\|_{0,\Gamma}^2+\frac{\alpha}{2}\|q\|_{0,\omega}^2\label{Poisson_control}
\end{eqnarray}
subject to
\begin{equation}\label{Poisson_Neumann}
\left\{
\begin{aligned}
-\Delta u+u&= f+q(x)\chi_\omega  \quad&\mbox{in}\ \Omega,\\
\partial_nu&=u_n \quad &\mbox{on}\ \Gamma,
\end{aligned}
\right.
\end{equation}
where $\alpha>0$ is the regularization parameter. This approach has been extensively studied in the literature (cf. \cite{ChengGongHanZheng,HanCongWang}). Introduce the following adjoint equation: find $p\in H^1(\Omega)$ such that
\begin{equation}\label{Poisson_adjoint}
\left\{
\begin{aligned}
-\Delta p+p&= 0 \quad&\mbox{in}\ \Omega,\\
\partial_np&=u-u_d \quad &\mbox{on}\ \Gamma.
\end{aligned}
\right.
\end{equation}
The first order optimality condition then implies 
\begin{eqnarray}
\nabla J_0(u,q)=\alpha q+p=0\quad \mbox{in}\  \omega,\label{opt} 
\end{eqnarray}
which in turn gives $q=-\frac{1}{\alpha}p|_{\omega}$. The well-posedness analysis, finite element approximation, and variations of the above formulation have been extensively studied in, for example, \cite{ChengGongHanZheng, HanCongWang}.

In this note, we intend to identify both $\omega$ and $q(x)$. Following a similar approach, we formulate the optimization problem as follows:
\begin{eqnarray}
\min\limits_{\omega\subset\Omega,\ q\in L^2(\omega)}\ J_0(\omega,u,q)=\frac{1}{2}\|u-u_d\|_{0,\Gamma}^2+\frac{\alpha}{2}\|q\|_{0,\omega}^2 \quad \mbox{subject\ to}\ \eqref{Poisson_Neumann}.\label{shape_control_obj}
\end{eqnarray}
In contrast to the problem \eqref{Poisson_control}-\eqref{Poisson_Neumann}, here both $\omega$ and $q$ are unknown. However, assuming $\omega$ is given, we are led to the standard inverse source problem \eqref{Poisson_control}-\eqref{Poisson_Neumann}. Replacing the source $q$ by $-\frac{1}{\alpha}p$, we transform the optimization problem \eqref{shape_control_obj} into the following shape optimization problem:
\begin{eqnarray}
\min\limits_{\omega\subset\Omega} \ J_0(\omega,u,p)=\frac{1}{2}\|u-u_d\|_{0,\Gamma}^2+\frac{1}{2\alpha}\|p\|_{0,\omega}^2 \label{shape_obj}
\end{eqnarray}
subject to
\begin{equation}\label{shape_state}
\left\{
\begin{aligned}
-\Delta u+u&= f-\frac{1}{\alpha}p\chi_\omega  \quad&\mbox{in}\ \Omega,\\
\partial_nu&=u_n \quad &\mbox{on}\ \Gamma,\\
-\Delta p+p&= 0 \quad&\mbox{in}\ \Omega,\\
\partial_np&=u-u_d \quad &\mbox{on}\ \Gamma.
\end{aligned}
\right.
\end{equation}
Thus, we eliminate the unknown source $q$ and reformulate the optimization problem over $(q, \omega)$ with the state $u$ into an optimization problem over $\omega$ with the state $(u, p)$. The remaining task is to study the well-posedness and derive the shape sensitivity analysis for the above shape optimization problem.

The weak formulation of the coupled system (\ref{shape_state}) is: find $(u,p)\in H^1(\Omega)\times H^1(\Omega)$ such that
\begin{equation}\label{shape_state_weak}
\left\{
    \begin{aligned}
        &\int_{\Omega}(\nabla u \cdot \nabla v+ uv )\mathrm{dx} + \frac{1}{\alpha} \int_{\omega}pv \mathrm{dx} = \int_{\Omega}fv \mathrm{dx} + \int_{\Gamma}u_n v \mathrm{ds}&\quad &\forall v\in H^1(\Omega),\\
        &\int_{\Omega}(\nabla p \cdot \nabla w+ pw) \mathrm{dx} = \int_{\Gamma}(u-u_d) w \mathrm{ds}&\quad &\forall w\in H^1(\Omega).
    \end{aligned}
    \right.
\end{equation}
Since this system represents the first-order optimality conditions of the optimization problem \eqref{Poisson_control}-\eqref{Poisson_Neumann}, and given that the optimization problem is strictly convex and admits a unique solution, we conclude that the system \eqref{shape_state_weak} also admits a unique solution pair  $(u,p)\in H^1(\Omega)\times H^1(\Omega)$.

To derive a stability result for the coupled system \eqref{shape_state}, which will be used to prove the differentiability with respect to the domain $\omega$, we rescale the system and write the variational formulation as a symmetric saddle point system (cf. \cite{Gaspoz} for a similar approach). Define
\begin{eqnarray}
\bm{X}:=H^1(\Omega)\times H^1(\Omega)\quad\mbox{with}\ \ \|\bm{v}\|:=(\|v_1\|_{1,\Omega}^2+\|v_2\|_{1,\Omega}^2)^{\frac{1}{ 2}} \quad \forall \bm{v}=(v_1,v_2)\in \bm{X}.\nonumber
\end{eqnarray}
Introduce the bilinear form $\bm{b}:\bm{X}\times \bm{X}\rightarrow\mathbb{R}$ defined as 
\begin{eqnarray}
\bm{b}(\bm{v},\bm{w}):=\bm{a}(\bm{v},\bm{w})+\frac{1}{\sqrt{\alpha}}\bm{c}(\bm{v},\bm{w}),\nonumber
\end{eqnarray}
where
\begin{equation}
\begin{aligned}
&\bm{a}(\bm{v},\bm{w}):=a(v_1,w_2)+a(w_1,v_2) &\quad &\forall \bm{v},\bm{w}\in \bm{X},\nonumber\\
&a(u,v):=(\nabla u,\nabla v)+(u,v) &\quad &\forall u,v\in H^1(\Omega),\nonumber\\
&\bm{c}(\bm{v},\bm{w}):=(v_2,w_2)_\omega-(v_1,w_1)_\Gamma&\quad &\forall \bm{v},\bm{w}\in \bm{X}.\nonumber
\end{aligned}
\end{equation}
Then the weak formulation (\ref{shape_state_weak}) can be rewritten in the following compact form:
\begin{eqnarray}
\mbox{find}\ \bm{x}\in \bm{X}\quad\mbox{s.t.}\ \bm{b}(\bm{x},\bm{\phi})=-\frac{1}{\sqrt{\alpha}}(u_d,\phi_1)_\Gamma+(f,\phi_2)+(u_n,\phi_2)_\Gamma\quad\forall \bm{\phi}\in \bm{X},\label{compact_form}
\end{eqnarray}
where $\bm{x}=(x_1,x_2)\in \bm{X}$ with $x_1=u$ and $x_2=\frac{1}{\sqrt{\alpha}}p$ solving (\ref{shape_state_weak}).  

For each $u\in H^1(\Omega)$, the trace theorem gives us  $\|u\|_{0,\Gamma}\leq M_I\|u\|_{1,\Omega}$ for some constant $M_I>0$. For the characteristic function $\chi_\omega:L^2(\omega)\rightarrow H^1(\Omega)^*$, we assume that there exists a constant $M_C>0$ such that $\|\chi_\omega q\|_{H^1(\Omega)^*}\leq M_C\|q\|_{0,\omega}$. Let $M=\max\{M_I,M_C\}$, we define the semi-norm and $\alpha$-dependent norm as follows (cf. \cite{Gaspoz}):
\begin{equation}
\begin{aligned}
&|\bm{v}|:=(\|v_1\|_{0,\Gamma}^2+\|v_2\|_{0,\omega}^2)^{\frac{1}{2}} &\quad &\forall \bm{v}\in \bm{X},\nonumber\\
&\|\bm{\phi}\|_\alpha:=\|\bm{\phi}\|+\frac{M}{\sqrt{\alpha}}|\bm{\phi}| &\quad &\forall \bm{\phi}\in \bm{X}.\nonumber
\end{aligned}
\end{equation}
It is clear that $|\bm{v}|\leq M\|\bm{v}\|$ for any $\bm{v}\in \bm{X}$ (cf. \cite[eq. (2.16)]{Gaspoz}). One can then prove that the bilinear form $\bm{b}$ satisfies the following continuity and inf-sup properties (cf. \cite[Theorem 2.1]{Gaspoz}): 
\begin{equation}
\begin{aligned}
&|\bm{b}(\bm{v},\bm{\phi})|\leq \|\bm{v}\|\|\bm{\phi}\|+\frac{M}{\sqrt{\alpha}}\|\bm{v}\||\bm{\phi}|\leq \|\bm{v}\|\|\bm{\phi}\|_\alpha &\quad &\forall \bm{v},\bm{\phi}\in \bm{X},\\
&\sup\limits_{\bm{v}\in \bm{X}}\bm{b}(\bm{v},\bm{\phi})\geq \frac{1}{\kappa}\|\bm{v}\|\|\bm{\phi}\|_\alpha &\quad &\forall 0\neq \bm{\phi}\in \bm{X},
\end{aligned}
\end{equation}
where $\kappa =\frac{1+2L}{1+L}(1+M(1+\frac{2M}{\sqrt{\alpha}}))$ and $L=\frac{M}{\sqrt{\alpha}}$. By combining these properties of $\bm{b}$ with the boundedness of the right-hand side in (\ref{compact_form}) we obtain
\begin{eqnarray}
\|u\|_{1,\Omega}+\frac{1}{\sqrt{\alpha}}\|p\|_{1,\Omega}\leq C(\|u_d\|_{0,\Gamma}+\|u_n\|_{0,\Gamma}+\|f\|_{0,\Omega}),\label{up_bounded}
\end{eqnarray}
where the constant $C>0$ may depend on $\alpha$, $L$ and $M$.

It is well-known that the shape optimization problems are generally severally ill-posed, that is, solutions may not exist, or if they do, they may not be unique (cf. \cite{Allaire,AllaireHenrot}). A common remedy method involves adding a regularization term to the objective functional. To implement this, we introduce the function
\begin{equation}\label{peri}
    \mathcal{P}_{\Omega}(\omega): = \sup \{ \int_{\omega}\mbox{div} \phi \ \text{dx}: \ \phi \in \mathcal{D}^1(\Omega, \mathbb{R}^d), \ \max\limits_{x \in \Omega} ||\phi(x)|| \leq 1 \},
\end{equation}
which acts as the perimeter constraint of $\omega$ in $\Omega$. For $\lambda>0$, consider the new optimization problem
\begin{equation}\label{OPT_obj}
    \min\limits_{\omega \subset\Omega}\ J({\omega},u,p) :=J_0({\omega},u,p) +\lambda \mathcal{P}_{\Omega}(\omega)\quad\mbox{subject\ to}\ \ (\ref{shape_state}).
\end{equation}
This formulation ensures the existence of a solution.
\begin{theorem}\label{exist}
The shape identification problem (\ref{OPT_obj}) admits a solution.
\end{theorem}

First we introduce the set of characteristic functions (cf. \cite{Sokolowski})
\begin{equation}
   \mbox{Char}(\Omega) = \{\chi \in L^2(\Omega): \ \chi(1- \chi) = 0 \ \  \mbox{a.e. in} \ \Omega \}.\nonumber
\end{equation}
Since for each non-empty subset $\omega\subset \Omega$, $\chi_\omega\in \mbox{ Char}(\Omega)$, and the elliptic system (\ref{shape_state}) admits a unique solution denoted by $(u_\omega,p_\omega)$, we can represent $(u(\chi),p(\chi))$ as the solution to (\ref{shape_state}) with $\chi\in \mbox{ Char}(\Omega)$.
\begin{lemma}
The mapping $\chi \rightarrow (u(\chi),p(\chi))$ is continuous from $\mbox{Char}(\Omega)$ to $H^1(\Omega)\times H^1(\Omega)$.
\end{lemma}

\begin{proof}
    Let $\omega_1, \omega_2$ be two measurable subsets of $\Omega$ with corresponding characteristic functions $\chi_1, \chi_2$. For each $\omega_i$, the elliptic system (\ref{shape_state}) admits a unique solution defined as in (\ref{shape_state_weak}), denoted by $(u_i,p_i)$ $(i=1,2)$. Moreover, the difference $(u_1-u_2,p_1-p_2)$ satisfies the following system:
    \begin{eqnarray}\nonumber
    \left\{
    \begin{aligned}
       &\int_{\Omega}(\nabla (u_1-u_2) \cdot \nabla v+ (u_1-u_2)v )\mathrm{dx} + \frac{1}{\alpha} \int_{\Omega}(\chi_1p_1-\chi_2p_2)v \mathrm{dx} =0 &\quad &\forall v\in H^1(\Omega),\\
        &\int_{\Omega}(\nabla (p_1-p_2) \cdot \nabla w+ (p_1-p_2)w) \mathrm{dx} = 0 &\quad &\forall w\in H^1(\Omega),
        \end{aligned}
        \right.
    \end{eqnarray}
    which is equivalent to 
    \begin{eqnarray}\nonumber
    \left\{
    \begin{aligned}
       &\int_{\Omega}(\nabla (u_1-u_2) \cdot \nabla v+ (u_1-u_2)v )\mathrm{dx} + \frac{1}{\alpha} \int_{\Omega}\chi_1(p_1-p_2)v \mathrm{dx} =\frac{1}{\alpha} \int_{\Omega}(\chi_2-\chi_1)p_2v \mathrm{dx},\\
        &\int_{\Omega}(\nabla (p_1-p_2) \cdot \nabla w+ (p_1-p_2)w) \mathrm{dx} = 0
        \end{aligned}
        \right.
    \end{eqnarray}
    for any $(v,w)\in H^1(\Omega)\times H^1(\Omega)$. 
    Since $H^1(\Omega)\hookrightarrow L^4(\Omega)$, we immediately obtain
    \begin{eqnarray}
    \int_{\Omega}(\chi_2-\chi_1)p_2v \mathrm{dx}&\leq& C\|\chi_1-\chi_2\|_{0,\Omega}\|p_2\|_{L^4( \Omega)}\|v\|_{L^4(\Omega)}\nonumber\\
    &\leq &C\|\chi_1-\chi_2\|_{0,\Omega}\|p_2\|_{1,\Omega}\|v\|_{1,\Omega}.\nonumber
    \end{eqnarray}
    Combining this with the stability result (\ref{up_bounded}), we get
    \begin{eqnarray}
\|u_1-u_2\|_{1,\Omega}+\frac{1}{\sqrt{\alpha}}\|p_1-p_2\|_{1,\Omega}\leq C\|\chi_1-\chi_2\|_{0,\Omega}\|p_2\|_{1,\Omega}\leq C\|\chi_1-\chi_2\|_{0,\Omega}\nonumber
\end{eqnarray}
for some constant $C>0$. This completes the proof.
\end{proof}

\begin{lemma}(\cite[Lemma 2.6]{Sokolowski})
The function $\chi_{\omega} \rightarrow \mathcal{P}_{\Omega}(\omega)$ is lower semi-continuous on the set $\mbox{Char}(\Omega) \subset L^2(\Omega)$.
\end{lemma}

To prove the existence theorem, we define a characteristic set (cf. \cite{Sokolowski})
\begin{equation}
    \mbox{Char}(\Omega, M) = \{ \chi_{\omega} \in \mbox{Char}(\Omega):\ \mathcal{P}_{\Omega}(\omega) \leq M \}\nonumber
\end{equation}
for some $M>0$. According to Proposition 2.8 in \cite{Sokolowski}, the set $\mbox{Char}(\Omega, M)$ is compact in $L^2(\Omega)$. Now we are ready to prove Theorem \ref{exist}.
\begin{proof}
    If  the perimeter $\mathcal{P}_{\Omega}(\omega)$ is not finite for a domain $\omega$, then the objective functional $J$ is also infinite. Thus, we only need to consider the finite case. Define $j(\omega):=J(\omega,u_\omega,p_\omega)$. Let $\{ \chi_n \}$ be a minimization sequence of $\omega_n$ that satisfies
    \begin{equation} 
       \lim_{n \rightarrow \infty} j(\chi_n) = \inf_{\chi \in \mbox{Char}(\Omega)} j(\chi).\label{lim}
    \end{equation}
    It can be shown that there exists a constant $M>0$ such that $\mathcal{P}_{\Omega}(\omega_n) \leq M$ for all $n\in \mathbb{N}^+$. Therefore, we can replace the admissible domain in (\ref{lim}) by $\mbox{Char}(\Omega, M)$, incorporating the perimeter constraint. Since the set $\mbox{Char}(\Omega, M)$ is compact and $\chi_n \in \mbox{Char}(\Omega, M)$, we can extract a subsequence $\{ \chi_{n_k} \}$ which converges to $\chi_0 \in \mbox{Char}(\Omega, M)$. Using the lower semi-continuity of $J$, we obtain
    \begin{equation}
        j(\chi_0) \leq  \liminf_{k \rightarrow \infty} j(\chi_{n_k}) =  \lim_{n \rightarrow \infty} j(\chi_n) = \inf_{\chi \in \mbox{Char}(\Omega, M)} j(\chi).\nonumber
    \end{equation}
    On the other hand, by definition we have
    \begin{equation}
        J(\chi_0) \geq \inf_{\chi \in \mbox{Char}(\Omega, M)} J(\chi),\nonumber
    \end{equation}
    which implies 
    \begin{equation}
        J(\chi_0) = \inf_{\chi \in \mbox{Char}(\Omega, M)} J(\chi),\nonumber
    \end{equation}
   this completes the proof.
\end{proof}

\section{Shape sensitivity analysis}
In this section, we will calculate the shape derivative of the shape optimization problem (\ref{shape_obj})-(\ref{shape_state}) using standard shape sensitivity analysis.

To begin with, we introduce several definitions. Assume that $\Omega$ is an open, bounded, and Lipschitz domain. Let ${\bm{V}} \in \bm{W}_0^{1,\infty}(\Omega)$ be a given velocity field. Using the perturbation of identity approach, we denote the perturbed domain for each $t > 0$ by
\begin{equation}
    \Omega_t := T_t(\Omega)[{\bm{V}}] = ({\rm Id} + t{\bm{V}})(\Omega).\nonumber
\end{equation}
For any function $u \in H^1(\Omega)$ with associated function $u_t \in H^1(\Omega_t)$ defined in $\Omega_t$,  we define the material derivative at $x \in \Omega$ as (cf. \cite{Sokolowski})
\begin{equation}
     \dot{u}(x) = \lim\limits_{t \rightarrow 0^+} \frac{u_t(T_t(x))-u(x)}{t},\nonumber
\end{equation}
while the shape derivative of $u$ at $x \in \Omega$ is given by (cf. \cite{Sokolowski})
\begin{equation}
    u'(x)= \dot{u}(x) - \nabla u(x) \cdot {\bm{V}}(x).\nonumber
\end{equation}

\begin{theorem}\label{exist_md}
Assume that $f,u_n,u_d$ are given functions such that $f \in L^2(\Omega)$, $u_n \in H^{1}(\Omega)$ and $u_d \in H^{1}(\Omega)$. Further assume that $\omega$ is of class $C^k$ ($k\geq 2$). Then for the velocity field ${\bm{V}} \in C([0,\epsilon]; \mathcal{D}^k(\Omega;\mathbb{R}^d))$, the weak material derivative of the elliptic system (\ref{shape_state}) in the direction ${\bm{V}}$ exists.
\end{theorem}
\begin{proof}
We denote by $(u_t, p_t) \in H^1(\Omega_t)\times H^1(\Omega_t)$ the weak solution of the coupled system (\ref{shape_state}) on the perturbed domain $\Omega_t$.   This solution satisfies the following variational system:
\begin{equation}\label{perturbed_st}
\left\{
    \begin{aligned}
        &\int_{\Omega_t}(\nabla u_t \cdot \nabla v_t + u_tv_t )\mathrm{dx} + \frac{1}{\alpha} \int_{\omega_t}p_tv_t \mathrm{dx} = \int_{\Omega_t}fv_t \mathrm{dx} + \int_{\Gamma_t}u_n v_t \mathrm{ds},\\
        &\int_{\Omega_t}(\nabla p_t \cdot \nabla w_t + p_tw_t) \mathrm{dx} = \int_{\Gamma_t}(u_t-u_d) w_t \mathrm{ds},
    \end{aligned}
    \right.
\end{equation}
 for all test functions $(v_t,w_t):=(v\circ T_t^{-1},w\circ T_t^{-1}) \in H^1(\Omega_t)\times H^1(\Omega_t)$, where $(v,w) \in H^1(\Omega)\times H^1(\Omega)$. Since ${\bm{V}}|_{\Gamma}=0$, we have $\Gamma_t=\Gamma$ and the boundary integrals in the above variational problem are well-defined. Define
 \begin{equation}
    u^t= u_t \circ T_t\in H^1(\Omega), \quad  p^t= p_t \circ T_t \in H^1(\Omega),\nonumber
\end{equation}
and transform $\Omega_t$ to the reference domain $\Omega$ by a change of variables. We obtain
\begin{equation}\label{w_state}
\left\{
    \begin{aligned}
        &\int_{\Omega} (A(t)\nabla u^t \cdot \nabla v + J(t)u^tv) \mathrm{dx} + \frac{1}{\alpha} \int_{\omega}J(t)p^tv \mathrm{dx} = \int_{\Omega}J(t)f\circ T_tv \mathrm{dx} + \int_{\Gamma}M(t)u_n\circ T_t v \mathrm{ds},\\
        &\int_{\Omega}(A(t)\nabla p^t \cdot \nabla w + J(t)p^tw )\mathrm{dx} = \int_{\Gamma}M(t)(u^t-u_d\circ T_t) w \mathrm{ds},\\
    \end{aligned}
    \right.
\end{equation}
for all test functions $(v,w) \in H^1(\Omega)\times H^1(\Omega)$, where 
\begin{equation}
    J(t) = |\det DT_t|, \quad A(t) = J(t)(DT_t)^{-1 \ *}(DT_t)^{-1}, \quad M(t) = J(t)|^{*}(DT_t)^{-1}{\bm{n}}|.\nonumber
\end{equation}
Again use ${\bm{V}}|_{\Gamma} = {\bf 0}$, we know that $T_t|_{\Gamma} = {\rm Id}$ and $M(t)|_{\Gamma} = 1$. Moreover, due to the continuity of $A(t)$ and $J(t)$, we can choose $\tau > 0$ small enough such that for all $0 \leq t < \tau$, there exist constants $0 < C_1 < C_2$ satisfying
\begin{equation}
    \forall x \in \mathbb{R}^d, \ C_1 \le J(t) \le C_2 \ \text{and} \ C_1 |x|^2 \le A(t) x \cdot x \le C_2 |x|^2.\nonumber
\end{equation}
By assumption, $f \in L^2(\Omega), u_n \in L^{2}(\Gamma)$ and $u_d \in L^{2}(\Gamma)$. Consequently, we can find $C_3 > 0$ such that 
\begin{equation}
   ||f\circ T_t||_{0,\Omega} \leq ||f||_{0,\Omega} \leq C_3, \ ||u_n||_{0, \Gamma} \leq C_3 ,\ ||u_d||_{0, \Gamma} \leq C_3.\nonumber
\end{equation}
Since $A(t)$, $J(t)$ and $M(t)$ are all uniformly bounded, similar to the proof of (\ref{up_bounded}), we can show that
\begin{equation}
    ||u^t||_{1,\Omega} \leq C, \quad ||p^t||_{1,\Omega} \leq C,\nonumber
\end{equation}
for some $C > 0$  and all $t \in [0, \tau)$, where $C$ may depend on $\alpha$ but not on $t$. Next, define 
 \begin{equation}
    z_u^t = \frac{u^t-u}{t}, \quad z_p^t = \frac{p^t-p}{t}.\nonumber
\end{equation}
We then have
\begin{equation} \label{m_d}
\left\{
    \begin{aligned}
        &\int_{\Omega}\nabla z^t_u \cdot \nabla v + z^t_u v \mathrm{dx} + \frac{1}{\alpha} \int_{\omega}z^t_p v \mathrm{dx} \\
        &= \frac{1}{t} \left( \int_{\Omega}\left( (J(t)f\circ T_t-f)v - (A(t) - I)\nabla u^t \cdot \nabla v - (J(t) - 1)u^t v \right) \mathrm{dx} - \frac{1}{\alpha} \int_{\omega}(J(t) - 1)p^t v \mathrm{dx} \right),\\
        &\int_{\Omega}\nabla z^t_p \cdot \nabla w + z^t_p w \mathrm{dx} = \int_{\Gamma}z^t_u w \mathrm{ds} - \frac{1}{t} \left(     
        \int_{\Omega}\left(  (A(t) - I)\nabla p^t \cdot \nabla w + (J(t) - 1)p^t w \right) \mathrm{dx} \right).
    \end{aligned}
    \right.
\end{equation}
It is known that (cf. \cite{Sokolowski})
\begin{equation}\label{mapping}
\begin{aligned}
    &\lim\limits_{t \rightarrow 0^+} \frac{A(t) - I}{t} = (\mbox{div} {\bm{V}})I - ^{*}D {\bm{V}} - D{\bm{V}},\\
    &\lim\limits_{t \rightarrow 0^+} \frac{J(t) - 1}{t} = \mbox{div} {\bm{V}}, \\
    &\lim_{t \rightarrow 0^+} \frac{J(t)f\circ T_t-f}{t} = \mbox{div} (f{\bm{V}}).
\end{aligned}
\end{equation}
Similar to the proof of (\ref{up_bounded}), one can show that $z_u^t$ and $z_p^t$ are uniformly bounded, i.e.
\begin{equation}
    ||z_u^t||_{1,\Omega} \leq C, \quad ||z_p^t||_{1,\Omega} \leq C, \nonumber
\end{equation}
for some constant $C>0$. Thus, we can extract a subsequence, still denoted by $(z_u^t,z_p^t)$, that converges weakly in $H^1(\Omega)$, i.e.,
\begin{equation}
    z_u^t \rightharpoonup \dot{u}, \quad z_p^t \rightharpoonup \dot{p} \quad  \mbox{in} \ H^1(\Omega)\quad \mbox{as}\ t\rightarrow 0^+.
\end{equation}
From the definition of the material derivative we conclude the result.
\end{proof}

Under the assumptions of Theorem \ref{exist_md}, we know that the weak material derivatives $\dot{u}(\Omega;{\bm{V}})$ and $\dot{p}(\Omega;{\bm{V}})$ exist in $H^1(\Omega)$. Furthermore, we assume $\nabla u \cdot {\bm{V}} \in H^1(\Omega)$ and $\nabla p \cdot {\bm{V}} \in H^1(\Omega)$ for all velocity fields ${\bm{V}} \in C([0, \epsilon]; \mathcal{D}^k(\Omega;\mathbb{R}^d))$ $(k\geq 2)$. This condition holds if $u,p \in H^2(\Omega)$, which can be ensured if $\Omega$ is a convex polyhedron or has a smooth boundary, by imposing additional higher regularity on $u_n$ and $u_d$. By definition the shape derivatives of $u$ and $p$ in the direction ${\bm{V}}$ exist and lie in the space $H^1(\Omega)$. 

We now derive the forms of $u'(\Omega; {\bm{V}})$ and $p'(\Omega; {\bm{V}})$. For a function $\phi(t, x) \in C([0, \tau];W^{1,1}_{\rm loc}(\mathbb{R}^{d})) \cap C^{1}([0, \tau];L^{1}_{\rm loc}(\mathbb{R}^{d}))$, we define the function $F$ as follows:
\begin{equation}
    F_{\bm{V}}(t) := \int_{\Omega_t} \phi(t,x) \mathrm{dx}.
\end{equation}
The shape derivative of $F_{\bm{V}}(t)$ at $t=0$  can be calculated as (cf. \cite[Chap. 9, Thm. 4.2]{DelfourZolésio2011})
\begin{equation}\label{shapederi}
    dF(\Omega;\bm{V}) = dF_{\bm{V}}(0) = \int_{\Omega} (\phi'(0,x) + \mathrm{div} (\phi(0,x) \bm{V})) \mathrm{dx}.
\end{equation}
Considering the weak formulation (\ref{perturbed_st}) of the state system defined on the perturbed domain $\Omega_t$, we take the derivative with respect to $t$ at $t = 0$ on both sides. From (\ref{shapederi}) we have
\begin{equation}
\left\{
    \begin{aligned}
        &\int_{\Omega}(\nabla u' \cdot \nabla v + u'v) \mathrm{dx} + \frac{1}{\alpha} \int_{\omega}p'v \mathrm{dx} + \int_{\Omega}
        \mathrm{div}((\nabla u \cdot  \nabla v + uv)\bm{V}) \mathrm{dx} \\
        &+ \frac{1}{\alpha} \int_{\omega} \mathrm{div}(pv\bm{V}) \mathrm{dx} - \int_{\Omega} \mathrm{div}(fv\bm{V}) \mathrm{dx} = 0,\\
        &\int_{\Omega}(\nabla p' \cdot \nabla w + p'w) \mathrm{dx} + \int_{\Omega} \mathrm{div}((\nabla p \cdot  \nabla w + pw)\bm{V}) \mathrm{dx} - \int_{\Gamma} u'w \mathrm{ds} = 0.
    \end{aligned}
    \right.
\end{equation}
Using Green's formula, we can reformulate the system as: 
\begin{equation}
\left\{
    \begin{aligned}
        &\int_{\Omega}(\nabla u' \cdot \nabla v + u'v) \mathrm{dx} + \frac{1}{\alpha} \int_{\omega}p'v \mathrm{dx} + \int_{\Gamma} (\nabla u \cdot  \nabla v + uv)(\bm{V} \cdot \bm{n}) \mathrm{ds} \\
        &+ \frac{1}{\alpha} \int_{\partial\omega} pv(\bm{V} \cdot \bm{n}) \mathrm{ds} - \int_{\Gamma} fv(\bm{V} \cdot \bm{n}) \mathrm{ds} = 0,\\
        &\int_{\Omega}(\nabla p' \cdot \nabla w + p'w) \mathrm{dx} + 
        \int_{\Gamma} (\nabla p \cdot  \nabla w + pw)(\bm{V} \cdot \bm{n}) \mathrm{ds} - \int_{\Gamma} u'w \mathrm{ds} = 0,
    \end{aligned}
    \right.
\end{equation}
this along with the fact that ${\bm{V}}|_{\Gamma} = 0$ yields
\begin{equation}
\left\{
    \begin{aligned}
        &\int_{\Omega}(\nabla u' \cdot \nabla v + u'v) \mathrm{dx} + \frac{1}{\alpha} \int_{\omega}p'v \mathrm{dx} + \frac{1}{\alpha} \int_{\partial\omega}pv (\bm{V} \cdot \bm{n}) \mathrm{ds} = 0,\\ 
        &\int_{\Omega}(\nabla p' \cdot \nabla w + p'w) \mathrm{dx} - \int_{\Gamma}u'w \mathrm{ds} = 0. \\
    \end{aligned}
    \right.
\end{equation}
From this, we can show that the shape derivatives $(u', p')$ are given as the unique solution  to the following coupled boundary value problem:
\begin{equation}
\left\{
    \begin{aligned}
        -\Delta u' + u' + \frac{1}{\alpha} \chi_{\omega} p' &= 0 &\quad &\mbox{in} \ \Omega, \\
        -\Delta p' + p' &= 0  &\quad &\mbox{in} \ \Omega,\\
        \llbracket{\partial_n  u'} \rrbracket &= -\frac{1}{\alpha} p({\bm{V}} \cdot {\bm{n}})  &\quad &\mbox{on} \ \partial\omega,\\
        \partial_n  u' &= 0  &\quad &\mbox{on} \ \Gamma,\\
        \partial_n  p' &= u' &\quad &\mbox{on} \ \Gamma.
    \end{aligned}
    \right.
\end{equation}

Now, we are prepared to derive the shape gradient of the objective functional. In certain applications, the volume constraint is crucial, see e.g., problem (\ref{OCP_obj})-(\ref{OCP_constraint}). Therefore, we incorporate the volume constraint $|\omega| = \gamma_0$, where $0<\gamma_0 < |\Omega|$, into the shape optimization problem \eqref{OPT_obj}. We replace the perimeter constraint $\mathcal{P}_{\Omega}(\omega)$ with an integral over $\partial \omega$ and introduce the volume constraint as a penalty term in the objective functional (\ref{OPT_obj}). Consequently, the minimization problem is formulated as follows:
\begin{equation}\label{obj_fuc}
    \mathop{\min}_{\omega \subset \Omega} \ J(\omega,u,p) = \frac{1}{2}||u-u_d||^2_{0, \Gamma} + \frac{1}{2\alpha}||p||^2_{0,\omega} + \lambda \mathcal{P}_{\Omega}(\omega) + \beta (\int_{\Omega} \chi_{\omega} \text{dx} - \gamma_0),
\end{equation}
where $\lambda>0$ and $\beta\geq 0$. For the shape optimization problem (\ref{obj_fuc}), 
we formulate a second adjoint system. The coupled adjoint system in weak form is introduced as follows: 
find a couple $(v,w) \in H^1(\Omega)\times H^1(\Omega)$ satisfying
\begin{equation}\label{ads}
    \left\{
    \begin{aligned}
    &(\nabla v, \nabla r)_{\Omega} + (v,r)_{\Omega} = (w-u+u_d,r)_{\Gamma} &\quad &\forall r\in H^1(\Omega), \\
    &(\nabla w, \nabla s)_{\Omega} + (w,s)_{\Omega} = -\frac{1}{\alpha} (v+p, s)_{\omega} &\quad &\forall s\in H^1(\Omega).
    \end{aligned}
    \right.
\end{equation}
The adjoint $(v,w)$ satisfies the following coupled boundary value problem
\begin{equation}\label{ads_strong}
    \left\{
    \begin{aligned}
    - \Delta w+w&=-\frac{1}{\alpha} \chi_{\omega} (v+p) \quad &\mbox{in} \ \Omega, \\
    \partial_n w &=0 \quad &\mbox{on} \ \Gamma, \\
    - \Delta v+v&=0 \quad &\mbox{in} \ \Omega, \\
    \partial_n v &=w-u+u_d \quad &\mbox{on} \ \Gamma.
    \end{aligned}
    \right.
\end{equation}
By setting $w=0$ we have $v=-p$. Therefore, there is no need to solve the above adjoint system. 

For all $t \in [0, \tau]$ and $\phi_1, \phi_2, \psi_1, \psi_2 \in H^1(\Omega_t)$, we could construct the Lagrangian functional as follows:
\begin{equation}\label{lag}
\begin{aligned}
    \mathcal{L}(\Omega_t, \phi_1, \phi_2, \psi_1, \psi_2) &= \frac{1}{2} \int_{\Gamma_t} (\phi_1-u_d)^2 \mathrm{ds} + \frac{1}{2 \alpha} \int_{\omega_t} \phi_2^{2} \mathrm{dx} + \int_{\Omega_t} (\nabla \phi_1 \cdot \nabla \psi_1 +\phi_1 \psi_1-f\psi_1) \mathrm{dx}\\
    &+ \frac{1}{\alpha} \int_{\omega_t} \phi_2 \psi_1 \mathrm{dx} - \int_{\Gamma_t} u_n \psi_1 \mathrm{ds} +\int_{\Omega_t}(\nabla \phi_2 \cdot \nabla \psi_2 +\phi_2 \psi_2) \mathrm{dx}\\
    &- \int_{\Gamma_t} (\phi_1-u_d)\psi_2 \mathrm{ds} + \beta (\int_{\omega_t}  \mathrm{dx}-\gamma_0) + \lambda \int_{\partial\omega_t} \mathrm{ds}.
\end{aligned}
\end{equation}
Calculating the first order optimality condition of the Lagrangian functional (\ref{lag}), we obtain
\begin{eqnarray}\label{rlt}
    &&\frac{\partial \mathcal{L}(\Omega_t, u_t, p_t, v_t, w_t)}{\partial u_t}[\delta u] = \frac{\partial \mathcal{L}(\Omega_t,u_t, p_t, v_t, w_t)}{\partial p_t}[\delta p]=0,\nonumber\\
    &&\frac{\partial \mathcal{L}(\Omega_t, u_t, p_t, v_t, w_t)}{\partial v_t}[\delta v]=\frac{\partial \mathcal{L}(u_t, p_t, v_t, w_t)}{\partial w_t}[\delta w]= 0, \nonumber
\end{eqnarray}
for any $\delta u, \delta p, \delta v$ and $\delta w \in H^1(\Omega_t)$. From  these conditions, we deduce that the objective functional $J(\Omega_t)$ can be expressed as a min-max of the Lagrangian functional $\mathcal{L}$ with the saddle point $(u_t, p_t, v_t, w_t)$, i.e.,
\begin{equation}
    J(\Omega_t) = \min\limits_{(\phi_1, \phi_2)} \max\limits_{(\psi_1, \psi_2)} \mathcal{L}(\Omega_t, \phi_1, \phi_2, \psi_1, \psi_2). 
\end{equation}
To calculate the shape gradient through the Lagrangian $\mathcal{L}$, we introduce a parametrization of $H^1(\Omega_t)$ as follows:
\begin{equation}
    H^1(\Omega_t) = \{ \phi \circ T_t^{-1} : \phi \in H^1(\Omega) \}.\nonumber
\end{equation}
Using this parametrization we can reformulate the Lagrange functional $\mathcal{L}$ as:
\begin{equation}
    \Bar{\mathcal{L}}(t, \phi_1, \phi_2, \psi_1, \psi_2) =  \mathcal{L}(T_t(\Omega)[\bm{V}], \phi_1 \circ T_t^{-1}, \phi_2 \circ T_t^{-1}, \psi_1 \circ T_t^{-1}, \psi_2 \circ T_t^{-1}),
\end{equation}
with $\phi_1, \phi_2, \psi_1, \psi_2 \in H^1(\Omega)$. Let $(u,  p, v, w)$ be the solutions of the state and adjoint system. Then by C\'ea's method (cf. \cite{Cea}), we have
\begin{equation}\label{Cea_method}
    dJ(\Omega; \bm{V}) = \min\limits_{(\phi_1, \phi_2)} \max\limits_{(\psi_1, \psi_2)} \partial_t \Bar{\mathcal{L}}(t, \phi_1, \phi_2, \psi_1, \psi_2)|_{t=0} = \partial_t \Bar{\mathcal{L}}(t, u,  p, v, w)|_{t=0},
\end{equation}
for ${\bm{V}} \in \mathcal{U}$ with $\mathcal{U} = \{\bm{g} \in \bm{W}^{1, \infty}(\Omega):\ \bm{g}|_{\Gamma} = {\bf 0}\}$. As pointed out in Section 4.6 of \cite{Allaire}, the formal C\'ea's method is also rigorous if we can prove the shape differentiability of the state equation with respect to the domain, as done in Theorem \ref{exist_md}. Next, we calculate the Eulerian derivative of the shape functional. Note that ${\bm{V}}|_{\Gamma} = 0$, we rewrite the Lagrangian $\Bar{\mathcal{L}}$ defined on the perturbed domain $\Omega_t$ to one on the fixed domain $\Omega$
\begin{equation}\label{lt}
\begin{aligned}
    \Bar{\mathcal{L}}(t, &\phi_1, \phi_2, \psi_1, \psi_2) = \frac{1}{2} \int_{\Gamma_t} (\phi_1 \circ T_t^{-1}-u_d)^2 \mathrm{ds} + \frac{1}{2 \alpha} \int_{\omega_t} (\phi_2 \circ T_t^{-1})^{2} \mathrm{dx} - \int_{\Gamma_t} u_n \psi_1 \circ T_t^{-1} \mathrm{ds}\\
    &+ \int_{\Omega_t} (\nabla (\phi_1 \circ T_t^{-1}) \cdot \nabla (\psi_1 \circ T_t^{-1}) +\phi_1 \circ T_t^{-1} \psi_1 \circ T_t^{-1}-f\psi_1 \circ T_t^{-1}) \mathrm{dx} \\
    &+\int_{\Omega_t}(\nabla (\phi_2 \circ T_t^{-1}) \cdot \nabla (\psi_2 \circ T_t^{-1}) +\phi_2 \circ T_t^{-1} \psi_2 \circ T_t^{-1}) \mathrm{dx} - \int_{\Gamma_t} (\phi_1 \circ T_t^{-1}-u_d)\psi_2 \circ T_t^{-1} \mathrm{ds}\\
    &+ \frac{1}{\alpha} \int_{\omega_t} \phi_2 \circ T_t^{-1} \psi_1 \circ T_t^{-1}\mathrm{dx} + \beta (\int_{\omega_t} \mathrm{dx}-\gamma_0) + \lambda \int_{\partial\omega_t} \mathrm{ds} \\
    &= \frac{1}{2} \int_{\Gamma} (\phi_1-u_d)^2 \mathrm{ds} + \frac{1}{2 \alpha} \int_{\omega} J(t) (\phi_2)^{2} \mathrm{dx} + \int_{\Omega} A(t)(\nabla \phi_1 \cdot \nabla  \psi_1)\mathrm{dx} - \int_{\Gamma}u_n \psi_1 \mathrm{ds} \\
    &+  \int_{\Omega} J(t)  (\phi_1 \psi_1-\psi_1f\circ T_t) )\mathrm{dx} + \int_{\Omega}A(t) (\nabla \phi_2 \cdot \nabla \psi_2) \mathrm{dx}+ \int_{\Omega} 
     J(t)(\phi_2 \psi_2) \mathrm{dx} \\
    &+ \frac{1}{\alpha} \int_{\omega} J(t)\phi_2 \psi_1 \mathrm{dx} - \int_{\Gamma}(\phi_1-u_d)\psi_2 \mathrm{ds} + \beta (\int_{\omega} J(t) \mathrm{dx}-\gamma_0) + \lambda \int_{\partial\omega} M(t) \mathrm{ds}.
\end{aligned}
\end{equation}
Calculating the partial derivative of (\ref{lt}) with respect to $t$, we obtain 
\begin{equation}
\begin{aligned}
    \partial_t \Bar{\mathcal{L}}(t, \phi_1, \phi_2, &\psi_1, \psi_2)  
    = \frac{1}{2 \alpha} \int_{\omega} J'(t) (\phi_2)^{2} \mathrm{dx} + \int_{\Omega} A'(t)(\nabla \phi_1 \cdot \nabla \psi_1)\mathrm{dx}-\int_{\Omega} J(t)  (\psi_1 (\nabla f \cdot \bm{V}) )\mathrm{dx} \\
    &+  \int_{\Omega} J'(t)  (\phi_1\psi_1-\psi_1f\circ T_t) )\mathrm{dx} + \frac{1}{\alpha} \int_{\omega} J'(t) \phi_2\psi_1 \mathrm{dx} \\
    &+\int_{\Omega}A'(t) (\nabla \phi_2 \cdot \nabla \psi_2) \mathrm{dx}+ \int_{\Omega} 
     J'(t)(\phi_2\psi_2) \mathrm{dx} + \beta \int_{\omega} 
     J'(t) \mathrm{dx} + \lambda \int_{\partial\omega} M'(t) \mathrm{ds}.
\end{aligned}
\end{equation}
On the other hand, it has been proved that (cf. \cite[Chap. 9, eq. (4.12)]{DelfourZolésio2011})
\begin{equation}
\lim_{t \rightarrow 0^+} \frac{M(t)-1}{t} = \mbox{div} \bm{V} - D \bm{V} \bm{n} \cdot \bm{n}.
\end{equation}
Combining this with (\ref{mapping}) and (\ref{Cea_method}), we obtain
\begin{equation}
    \begin{aligned}
        dJ(\Omega;{\bm{V}}) = \partial_t \Bar{\mathcal{L}}(t, u, &p, v, q)|_{t=0}
        = \frac{1}{2 \alpha} \int_{\omega} \mbox{div} {\bm{V}} (p)^{2} \mathrm{dx} + \int_{\Omega} ((\mbox{div} {\bm{V}})I - ^{*}D {\bm{V}} - D{\bm{V}})(\nabla u \cdot \nabla v)\mathrm{dx} \\
        &+  \int_{\Omega} \mbox{div} {\bm{V}} (uv-vf) )\mathrm{dx} -\int_{\Omega} (v (\nabla f \cdot {\bm{V}}) )\mathrm{dx} + \frac{1}{\alpha} \int_{\omega} \mbox{div} {\bm{V}} pv \mathrm{dx}\\
        & +\int_{\Omega}((\mbox{div} {\bm{V}})I - ^{*}D {\bm{V}} - D{\bm{V}}) (\nabla p \cdot \nabla q) \mathrm{dx}+ \int_{\Omega} \mbox{div}{\bm{V}} (pq) \mathrm{dx} \\
        &+ \lambda \int_{\partial\omega} (\mbox{div} \bm{V} - D \bm{V} \bm{n} \cdot \bm{n}) \mathrm{ds} +\beta \int_{\omega} \mbox{div} \bm{V} \mathrm{dx}.
    \end{aligned}
\end{equation}
This gives us the shape derivative of the distributed type.

Assume $\omega \in {C}^2$, we can derive the surface expression of the shape gradient as
\begin{equation}\label{bsg}
    \begin{aligned}
        dJ(\Omega;{\bm{V}})
        &= \int_{\Gamma} (\nabla u \cdot \nabla v +uv-v f)({\bm{V}} \cdot {\bm{n}}) \mathrm{ds} + \int_{\Gamma} (\nabla p \cdot \nabla +pq)({\bm{V}} \cdot {\bm{n}}) \mathrm{ds} \\
             &+\frac{1}{2 \alpha} \int_{\partial\omega} p^2({\bm{V}} \cdot {\bm{n}}) \mathrm{ds} + \frac{1}{\alpha} \int_{\partial\omega} pv({\bm{V}} \cdot {\bm{n}}) \mathrm{ds} + \frac{1}{2}\int_{\Gamma} (\frac{\partial}{\partial {\bm{n}}} + \kappa)((u_t-u_d )^2)({\bm{V}} \cdot {\bm{n}}) \mathrm{ds}\\
        &-\int_{\Gamma} (\frac{\partial}{\partial {\bm{n}}} + \kappa)(u_{\bm{n}} v)({\bm{V}} \cdot \bm{n}) \mathrm{ds}-\int_{\Gamma} (\frac{\partial}{\partial {\bm{n}}} + \kappa)((u-u_d)q)({\bm{V}} \cdot {\bm{n}}) \mathrm{ds} \\
        &+ \lambda \int_{\partial\omega} (\mbox{div} \bm{V} - D \bm{V} \bm{n} \cdot \bm{n}) \mathrm{ds} + \beta \int_{\partial\omega} ({\bm{V}} \cdot \bm{n}) \mathrm{ds}.
    \end{aligned}
\end{equation}
Since ${\bm{V}}|_{\Gamma} = 0$, (\ref{bsg})  simplifies to
\begin{equation}\label{bsg2}
    \begin{aligned}
        dJ(\Omega;\bm{V})
        &=\frac{1}{2 \alpha} \int_{\partial\omega} (p^2 + 2pv)({\bm{V}} \cdot \bm{n}) \mathrm{ds} + \lambda \int_{\partial\omega} (\mbox{div} \bm{V} - D \bm{V} \bm{n} \cdot \bm{n}) \mathrm{ds} +\beta \int_{\partial\omega} ({\bm{V}} \cdot \bm{n}) \mathrm{ds}.
    \end{aligned}
\end{equation}
This concludes the derivation of the shape derivative and its surface expression.

\section{Numerical experiments}
In this section, we present numerical experiments to demonstrate the effectiveness and efficiency of the proposed algorithm. In our numerical examples, we consider (\ref{obj_fuc}) and use the shape gradient descent method to solve the resulting shape optimization problem. This method is the most popular shape optimization algorithm due to its simplicity and convenience for integration with various domain representation and tracking algorithms. The key aspect of this algorithm is selecting the descent direction, in which the shape derivative plays a pivotal role. In this paper, we employ the Hilbertian regularization method (cf. \cite{Allaire}) to extend and smooth the shape gradient given in (\ref{bsg2}). Specifically, we seek $\bm{V}\in \bm{H}^1_0(\Omega)$ such that
\begin{eqnarray}\label{H1_regularization}
(\nabla \bm{V},\nabla \bm{W}) +(\bm{V},\bm{W})=-  dJ(\Omega;\bm{W})\quad\forall  \bm{W}\in \bm{H}^1_0(\Omega).
\end{eqnarray}

In numerical experiments, to handle large deformations or variations between the initial and optimal shapes, we use the level-set method to track the domain boundary. The level-set method, proposed by Osher and Sethian (\cite{OsherSethian}) to track free boundaries, has been subsequently used to solve shape and topology optimization problems (cf. \cite{AllaireJouveToader, WangWangGuo2003}). Given a domain $\omega \subset \Omega$ and fix time $t$, we can define a level-set function $\phi:\mathbb{R}^{d+1} \rightarrow \mathbb{R}$ of $\Omega$ as follows:
\begin{equation}\nonumber
\left\{
    \begin{aligned}
        &\ \phi(t, x) < 0 \quad \mbox{if} \ x \in \omega, \\
        &\ \phi(t, x) = 0 \quad \mbox{if} \ x \in \partial \omega, \\
        &\ \phi(t, x) > 0 \quad \mbox{if} \ x \in \Omega \backslash \bar{\omega}.
    \end{aligned}
    \right.
\end{equation}
According to the definition, we can represent the domain $\omega$ by the negative part of $\phi$ and evolve the level-set function using the velocity field $\bm{V}$. In practice, we use the following advection type equation to update the level-set function $\phi$:
\begin{equation}\label{lst_update}
\left\{
    \begin{aligned}
         \frac{\partial \phi}{\partial t}(t,x) + \bm{V}(t,x) \cdot \nabla \phi(t,x) &= 0  &\quad &\mbox{in}\ \Omega\times (0,T),\\
        \phi(0, x) &= \phi_0 &\quad &\mbox{in}\ \Omega.
    \end{aligned}
    \right.
\end{equation}

A special case of the level-set function is the signed distance function, which meets the requirement of being neither too flat nor too steep near the boundary. During the iteration, we should check if $\phi$ deviates significantly from the signed distance function and reinitialize the level-set function if necessary. We give the re-distancing equation to correct $\phi$:
\begin{equation}\label{reinit}
\left\{
    \begin{aligned}
         \frac{\partial \phi}{\partial t}(t,x) + \mbox{sgn}(\phi) (|\nabla \phi(t,x)|-1) &= 0 &\quad &\mbox{in}\ \Omega\times (0,T), \\
        \phi(0, x) &= \phi_0 &\quad &\mbox{in}\ \Omega.
    \end{aligned}
    \right.
\end{equation}

Combining the level-set representation of domains with the steepest descent method, we present the shape optimization algorithm below:
\begin{Algorithm}\label{Alg:5.1}\textbf{Shape steepest descent method}
\begin{enumerate}
\item \textbf{Require}: {$u_d$, $\phi_{0} $, $\alpha$, $\beta$, $\lambda$ }   
  \item  {$k \gets 0$}
  \item  \textbf{repeat}
  \item  \quad
solve the state system (\ref{shape_state_weak}). 
\item  \quad
solve the adjoint system (\ref{ads}). 
\item  \quad
compute the descent direction by solving the Hilbertian regularization equation (\ref{H1_regularization}).
\item \quad
 update the level-set function by solving (\ref{lst_update}).
\item \quad check if reinitialization is needed, if so, reinitialize by (\ref{reinit}).  
\item \quad $k\gets k + 1$
\item  
\textbf{until} {$\vert J(\Omega^{k+1}) - J(\Omega^k) \vert < \epsilon \vert J(\Omega^k)\vert$} 
\end{enumerate}
\end{Algorithm}

In the following numerical examples, we consider the two dimensional case and fix the hold-all domain $\Omega = \{(x,y) \in \mathbb{R}^2: |x| < 1, \ |y| <1 \}$. All numerical experiments are conducted using the open-source software NGSolve (cf. \cite{Schoberl}).

During the iterations, we fix the penalty parameter $\lambda$ for the perimeter constraint at $10^{-6}$. The choice of the optimal regularization parameter $\alpha$ is critical for the regularization strategy. In our experiments, the regularization parameter $\alpha$ and the volume penalty parameter $\beta$ are fixed for the first twenty steps and then decrease at a rate of $0.9$. The parameters for all examples are chosen in the same manner, with different initial values.

\begin{Example}\label{e1}
In the first example, we set $u_n = \sin(\pi x)\sin(\pi y)$ on $\Gamma$, $f = 1$ in $\Omega$, and the exact source function $q_e = 1$. The exact domain $\omega_e$ is defined as  $\{(x,y) \in \Omega: 10(x+0.4-y^2)^2+x^2+y^2 < 0.5 \}$. The data $u_d$ is computed from the solution of the elliptic equation (\ref{Poisson}) using $u_n$, $f$ and the exact source $q_e$ on a fine finite element mesh consisting of 41032 elements and 20783 vertices. 

We choose the initial domain as $\omega_0 = \{(x,y) \in \Omega: (x+0.1)^2 + y^2 < 0.04 \}$ and solve the shape optimization problem on a mesh with 5820 elements and 3011 vertices. Delaunay triangles are used to partition the domain, and Algorithm \ref{Alg:5.1} is employed to obtain the approximate domain $\omega$. Define 
\begin{equation}
    \mbox{err}(q) = \Big(\int_{\omega_e} (q-q_e)^2 \mathrm{dx}\Big)^{\frac{1}{2}},
\end{equation}
which measures the distance of the approximate solutions from the true function. For a similar measure of errors, we refer to \cite{ChengGongHanZheng}. 

\begin{figure}[!htbp]
  \centering
  \subfigure[objective \ functional \ $J$]{
  \includegraphics[width=3in]{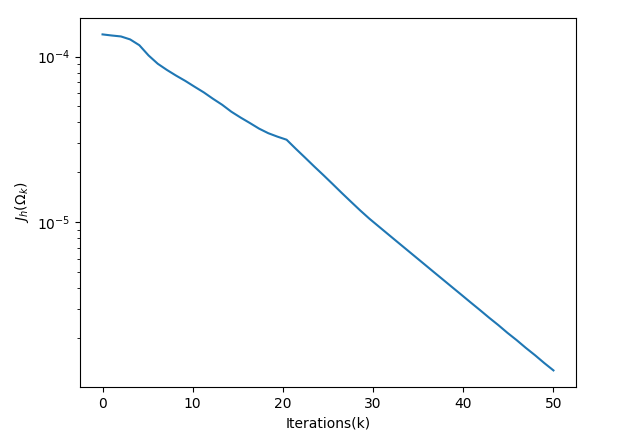}}
  \subfigure[$\mbox{err}(q)$]{
  \includegraphics[width=3in]{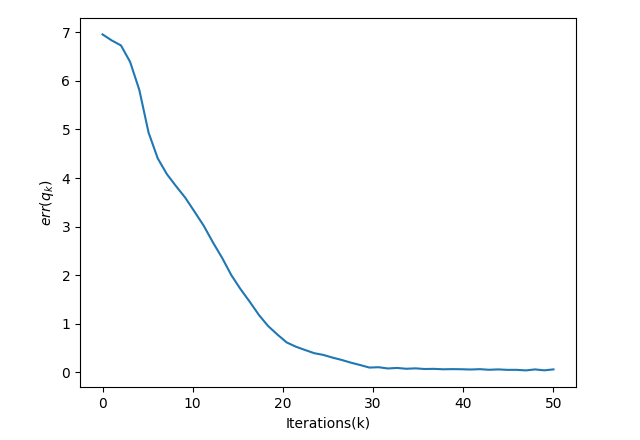}}
  \caption{Objective functional value and strength error for Example \ref{e1}.}
  \label{fig11}
\end{figure}

\begin{figure}[!htbp]
  \centering
  \subfigure[initial domain $\omega_0$]{
  \includegraphics[width=2.3in]{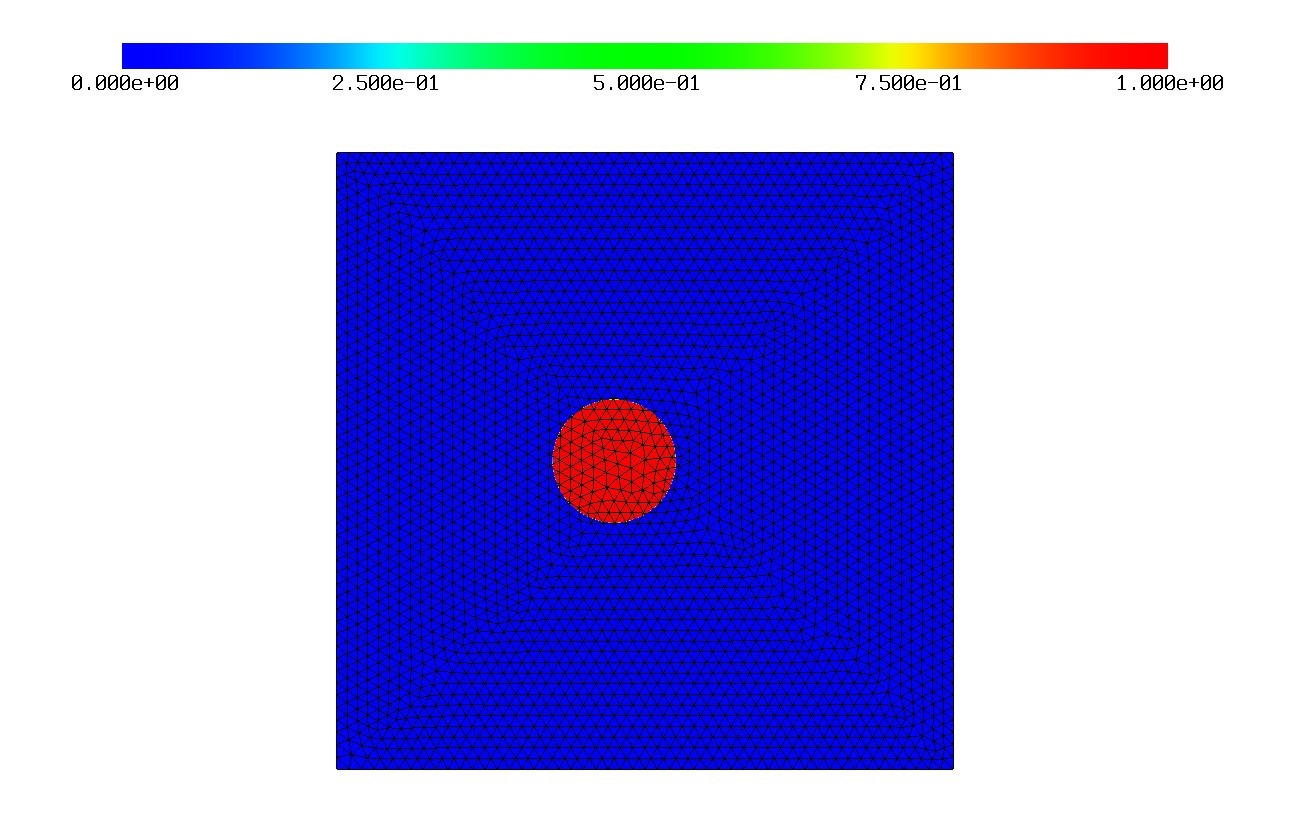}}\hspace{10mm}
  \subfigure[true source value]{
  \includegraphics[width=2.3in]{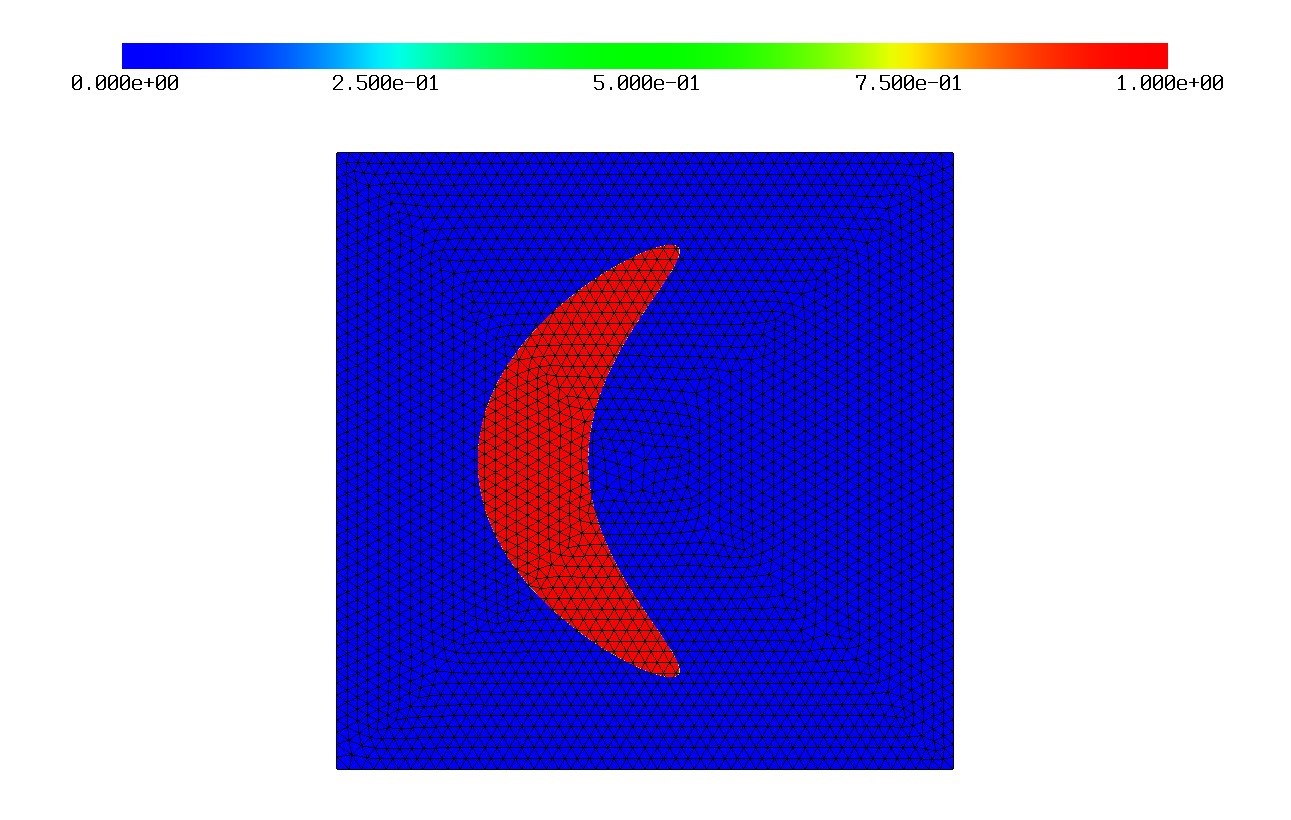}}
  \subfigure[reconstructed $\omega$ without noise]{
  \includegraphics[width=2.3in]{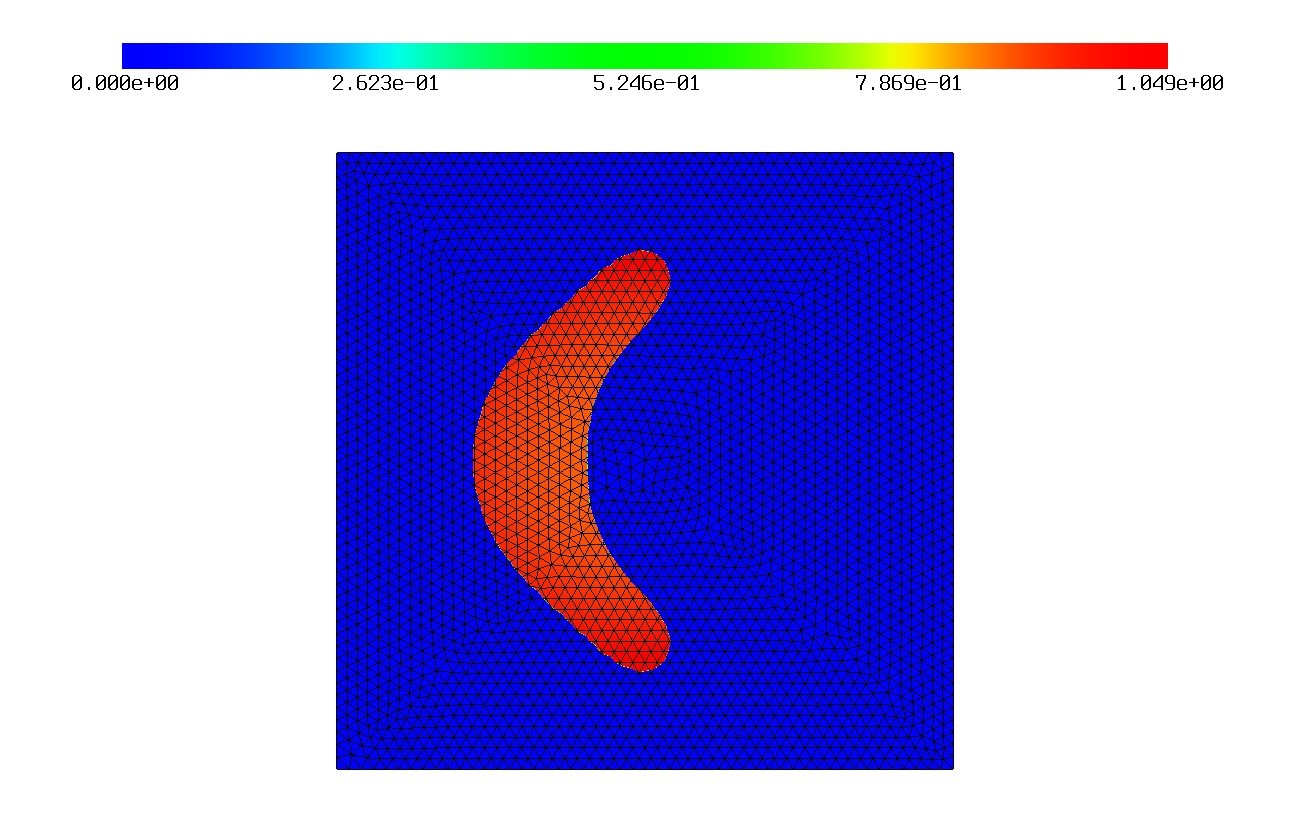}}\hspace{10mm}
  \subfigure[reconstructed $\omega$ with $10\%$  noise]{
  \includegraphics[width=2.3in]{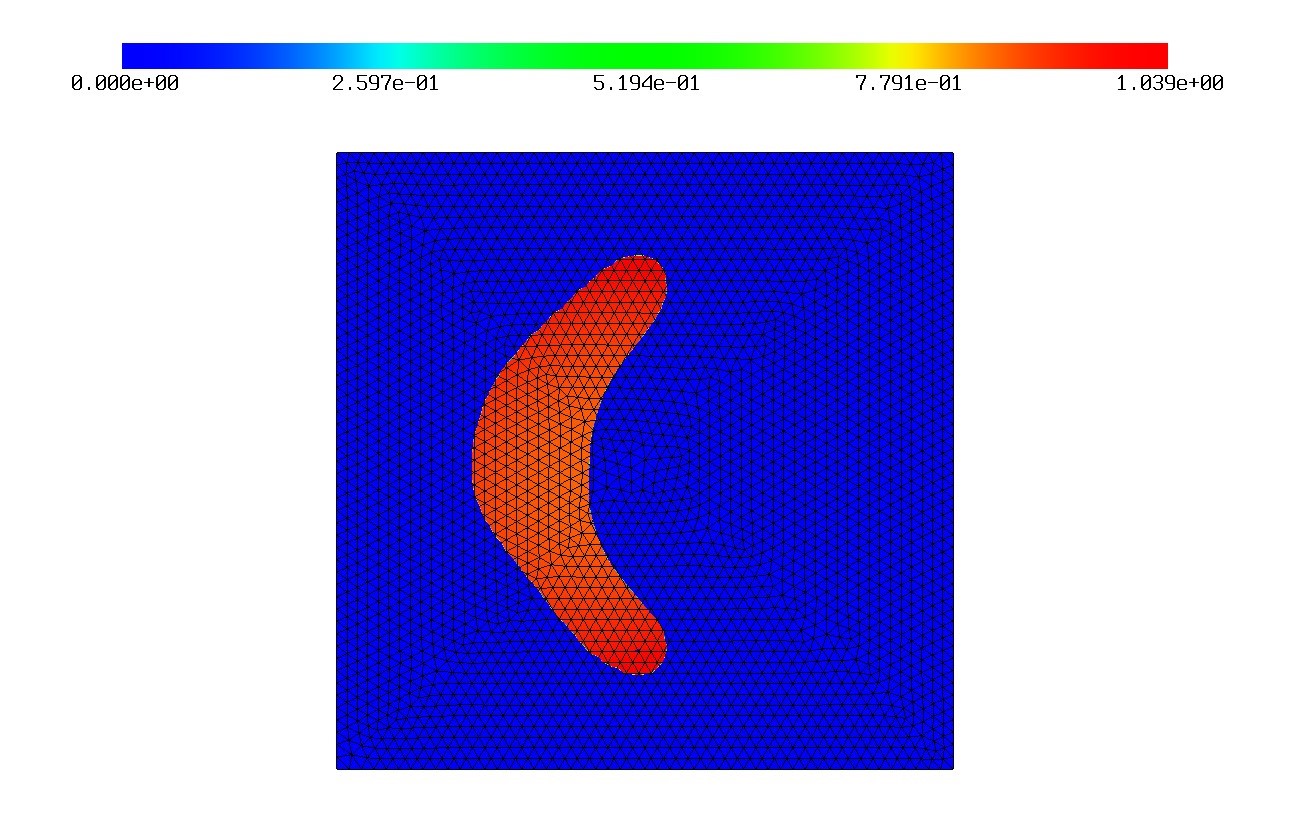}}
  \caption{Numerical results for Example \ref{e1} with different noise levels.}
  \label{fig12}
\end{figure}

Figure \ref{fig11} shows the decrease in the objective functional value alongside the strength error $\text{err}(q)$ as the iteration number increases. The reconstructed results are displayed in Figure \ref{fig12}. From Figures \ref{fig11} and \ref{fig12}, we observe that the numerical results closely approximate the true solution, accurately reconstructing both the support and the intensity of the source.

In practice, the observation $u_d$ may contain noise, so we test the algorithm's stability by adding noise to the boundary measurement $u_d$. Gaussian noise with a standard deviation $\sigma = 0.01$ is added to $u_d$, and the reconstructed domain $\omega$ is presented in Figure \ref{fig12}, showing satisfactory results. For complex shapes, the algorithm yields relatively good results. Figure \ref{fig12} illustrates that while the smooth parts of the domain are captured accurately, sharp corners are not fully recovered. This phenomenon is more precisely observed in the following example.

We remark that our proposed algorithm can simultaneously recover the support and intensity of the source by setting $q(x)=-\frac{1}{\alpha}p(x)|_{\omega}$. However, after the shape optimization algorithm, we can use a standard inverse source algorithm to refine the intensity $q(x)$ using the obtained support information $\omega$. This simple post-processing step allows the strength error $\text{err}(q)$ to reach a magnitude of $5.31 \times 10^{-2}$. 
\end{Example}

\begin{Example}\label{e2}
In the second example we consider a polygonal domain $\omega$. We set $u_n = \sin(\pi x)\sin(\pi y)$ on $\Gamma$, $f = 1$ in $\Omega$, and the exact source function $q_e = 1$. The exact domain $\omega_e$ is defined as $\{(x,y) \in \Omega: -0.1 < x < 0.6,\ 0.1 < y< 0.4 \}$. The data $u_d$ is computed in the same manner as above. 

\begin{figure}[!htbp]
  \centering
  \subfigure[initial domain $\omega_0$]{
  \includegraphics[width=2.5in]{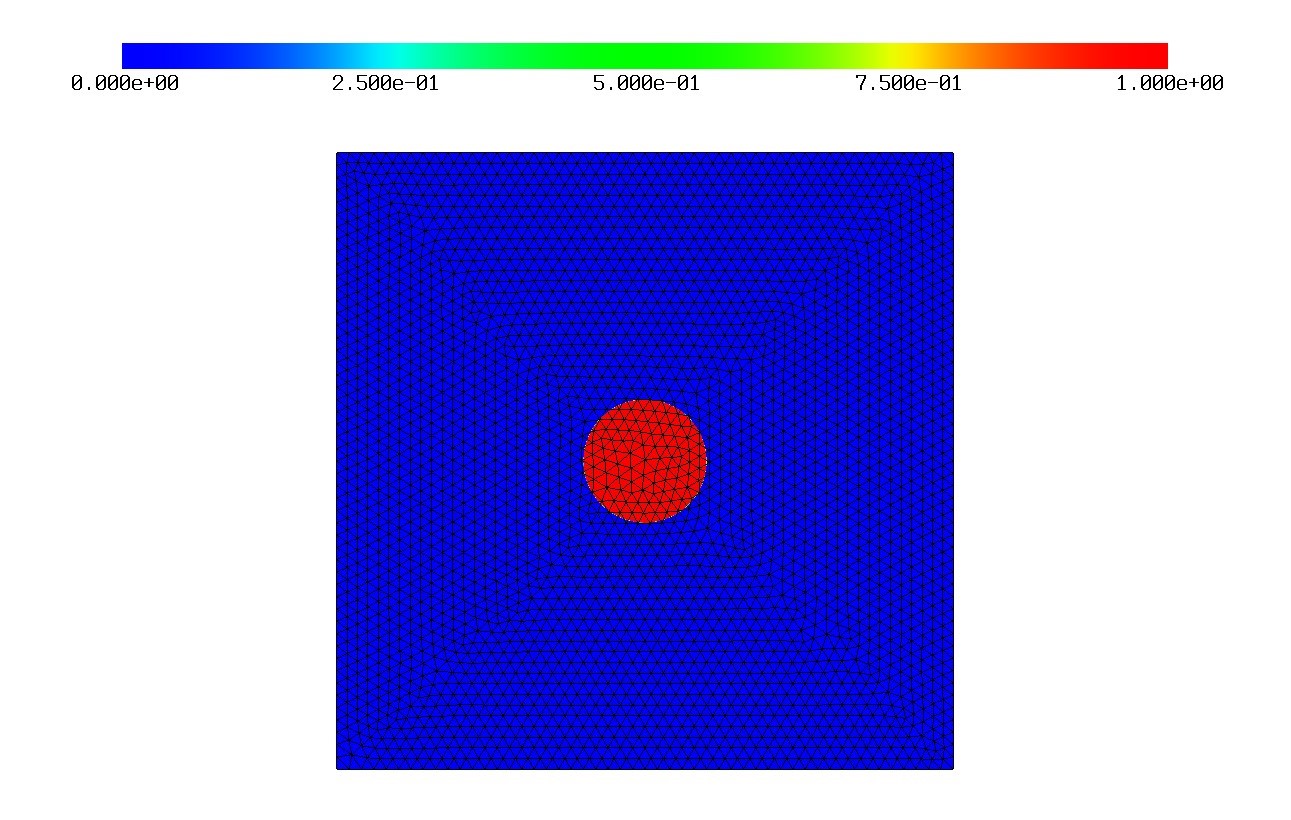}}\hspace{10mm}
  \subfigure[true source value]{
  \includegraphics[width=2.5in]{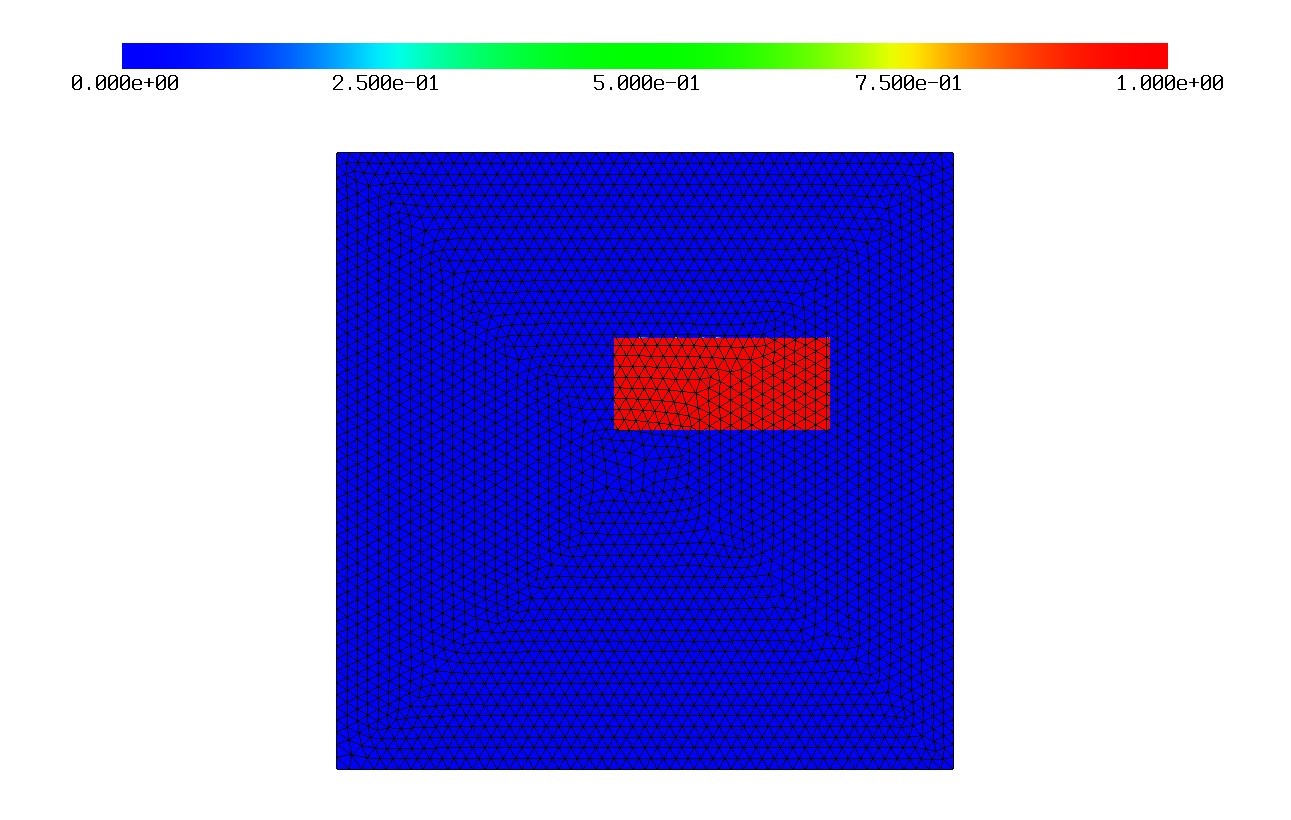}}
  \subfigure[reconstructed $\omega$ without noise]{
  \includegraphics[width=2.5in]{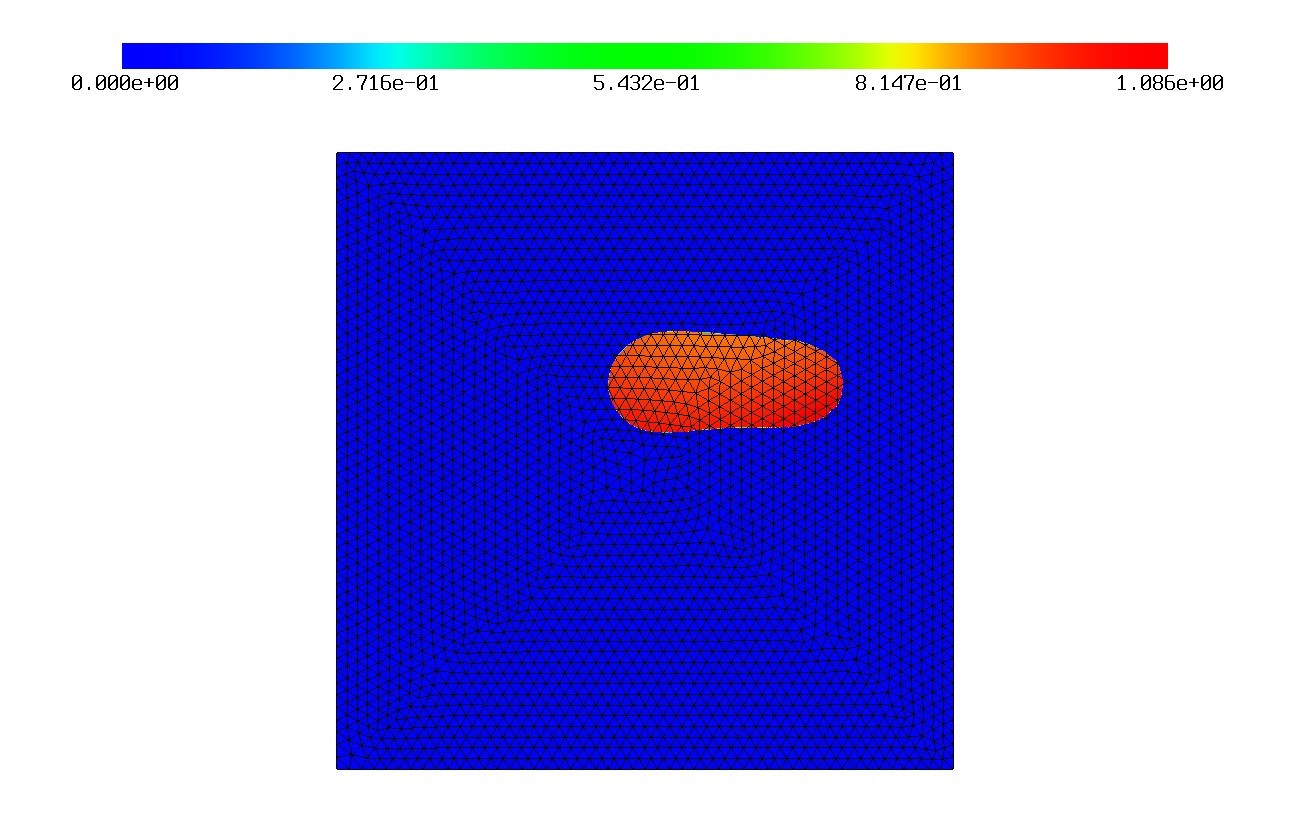}}\hspace{10mm}
  \subfigure[reconstructed $\omega$ with  $10\%$  noise]{
  \includegraphics[width=2.5in]{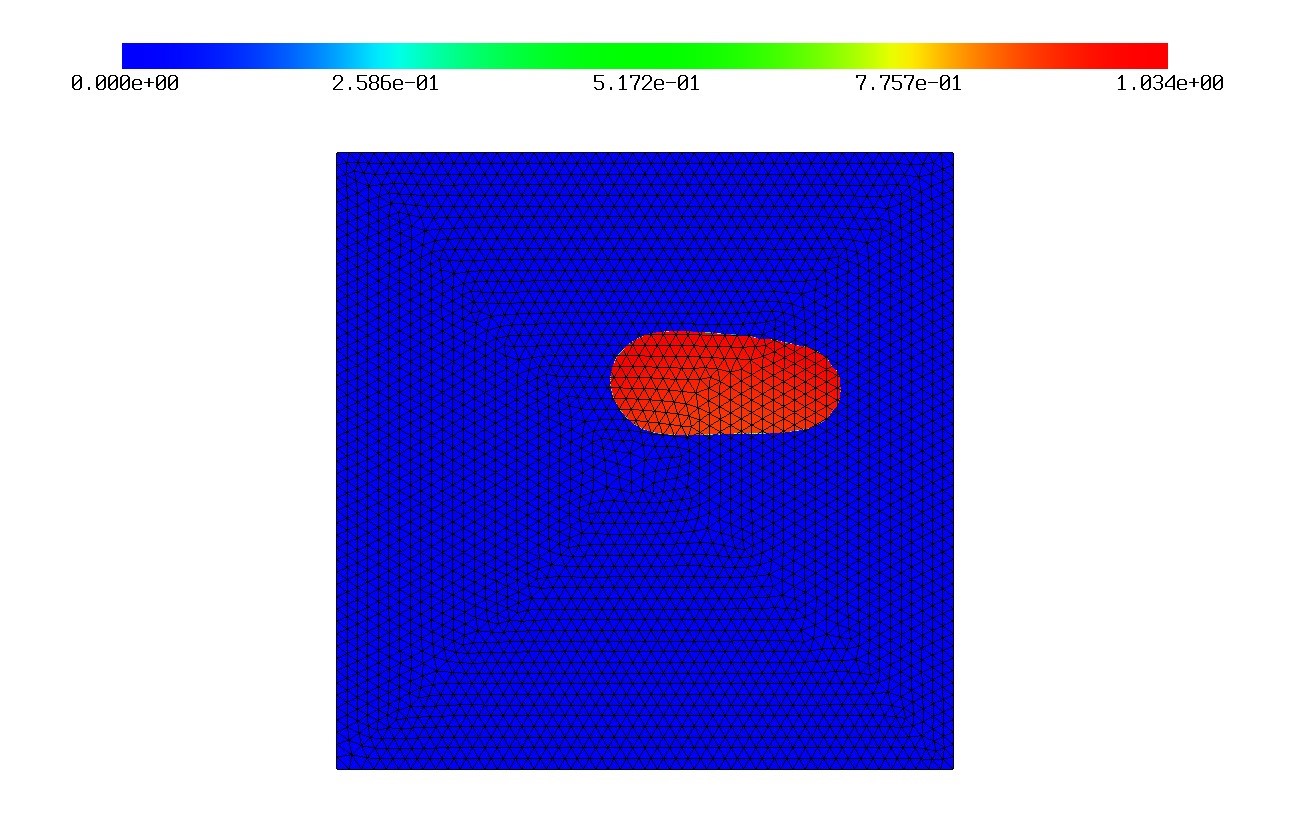}}
  \caption{Numerical results for Example \ref{e2} with different noise levels.}
  \label{fig2}
\end{figure}

We choose the initial domain as $\omega_0 = \{(x,y) \in \Omega: x^2 + y^2 < 0.04 \}$ with the same mesh size and finite element space as in Example \ref{e1}. The value of $\mbox{err}(q)$ is approximately $1.60 \times 10^{-2}$, and the reconstructed support and intensity are presented in Figure \ref{fig2}. We observe that for polygonal domains, the algorithm loses accuracy near the corners but results in a smoothed shape. This phenomenon is quite common in shape reconstruction algorithms and can be mitigated by using either the so-called $W^{1,\infty}$-velocity or  finer finite element meshes (cf. \cite{Deckelnick}).
\end{Example}

\begin{Example}\label{e3}
In the third example we consider a non-constant source, set $u_n = 1$ on $\Gamma$, $f = 1$ in $\Omega$, and the exact source function $q_e = \exp(-\frac{\sqrt{2}}{2} x) + \exp(-\frac{\sqrt{2}}{2} y)$. The exact domain $\omega_e$ is defined as $\{(x,y) \in \Omega: 36x^2 + 100 y^2 < 9 \}$. The data $u_d$ is computed in the same manner as above. 

We choose the initial domain as $\omega_0 = \{(x,y) \in \Omega: (x+0.3)^2 + (y+0.3)^2 < 0.04 \}$ with the same mesh size and finite element space as above. Figures \ref{fig31} and \ref{fig32} illustrate the convergence history of the objective functional and the reconstructed support and strength under different noise levels.

\begin{figure}[!htbp]
  \centering
  \subfigure[objective  functional  $J$]{
  \includegraphics[width=3in]{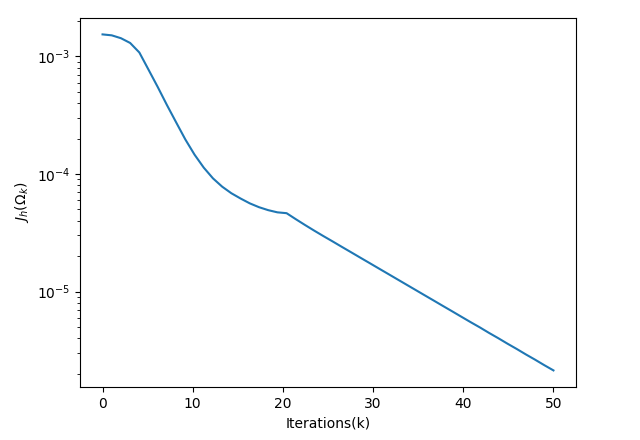}}
  \subfigure[$\mbox{err}(q)$]{
  \includegraphics[width=3in]{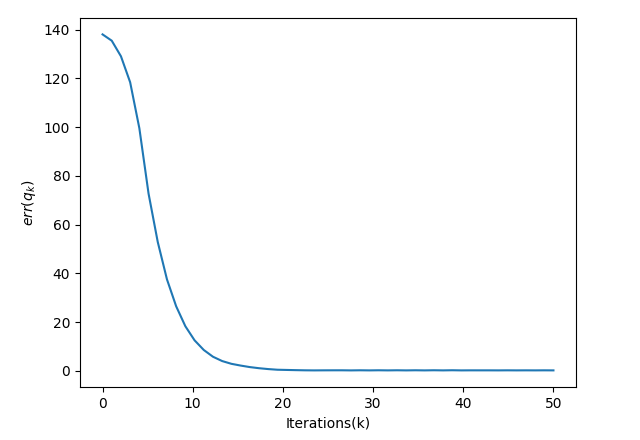}}
  \caption{Objective functional value and strength error for Example \ref{e3}.}
  \label{fig31}
\end{figure}

\begin{figure}[!htbp]
  \centering
  \subfigure[initial domain  $\omega_0$]{
  \includegraphics[width=2.5in]{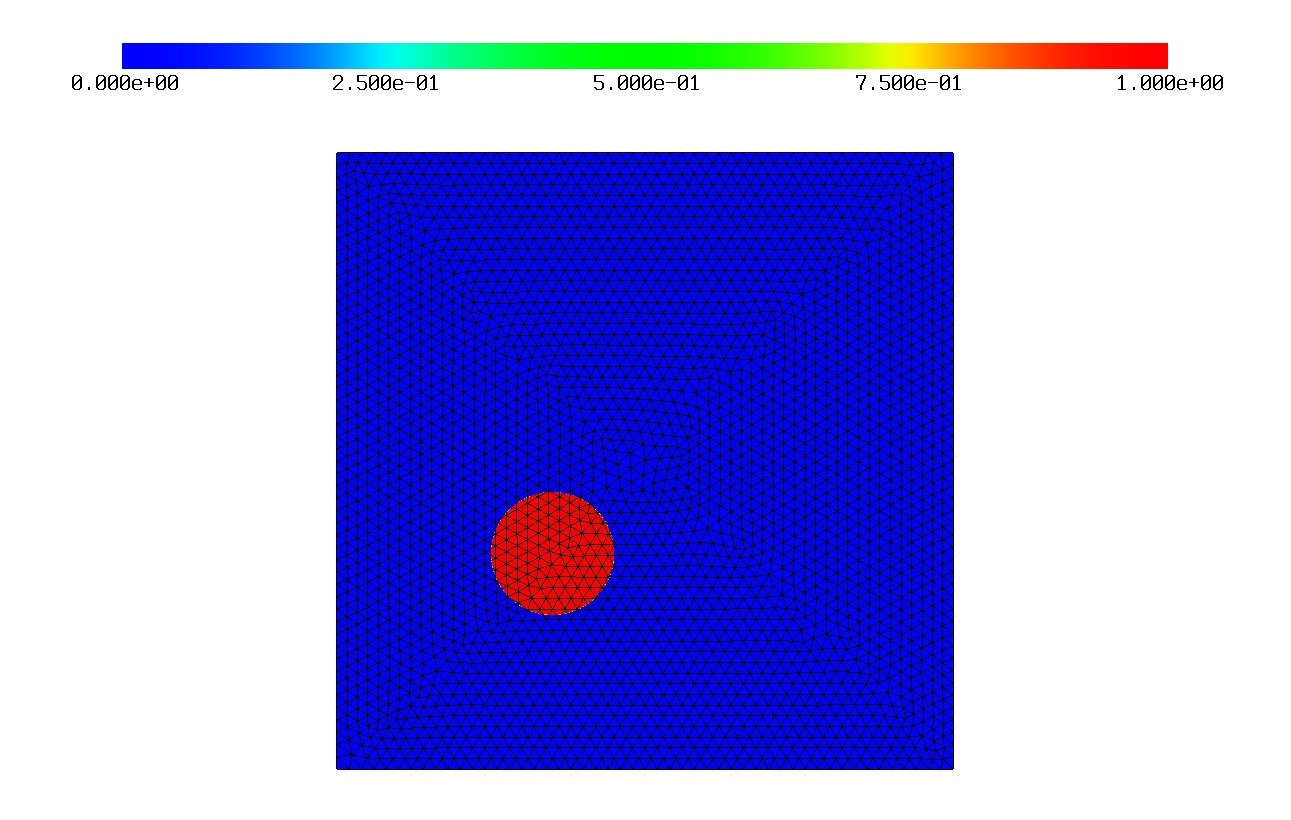}}\hspace{10mm}
  \subfigure[true source value]{
  \includegraphics[width=2.5in]{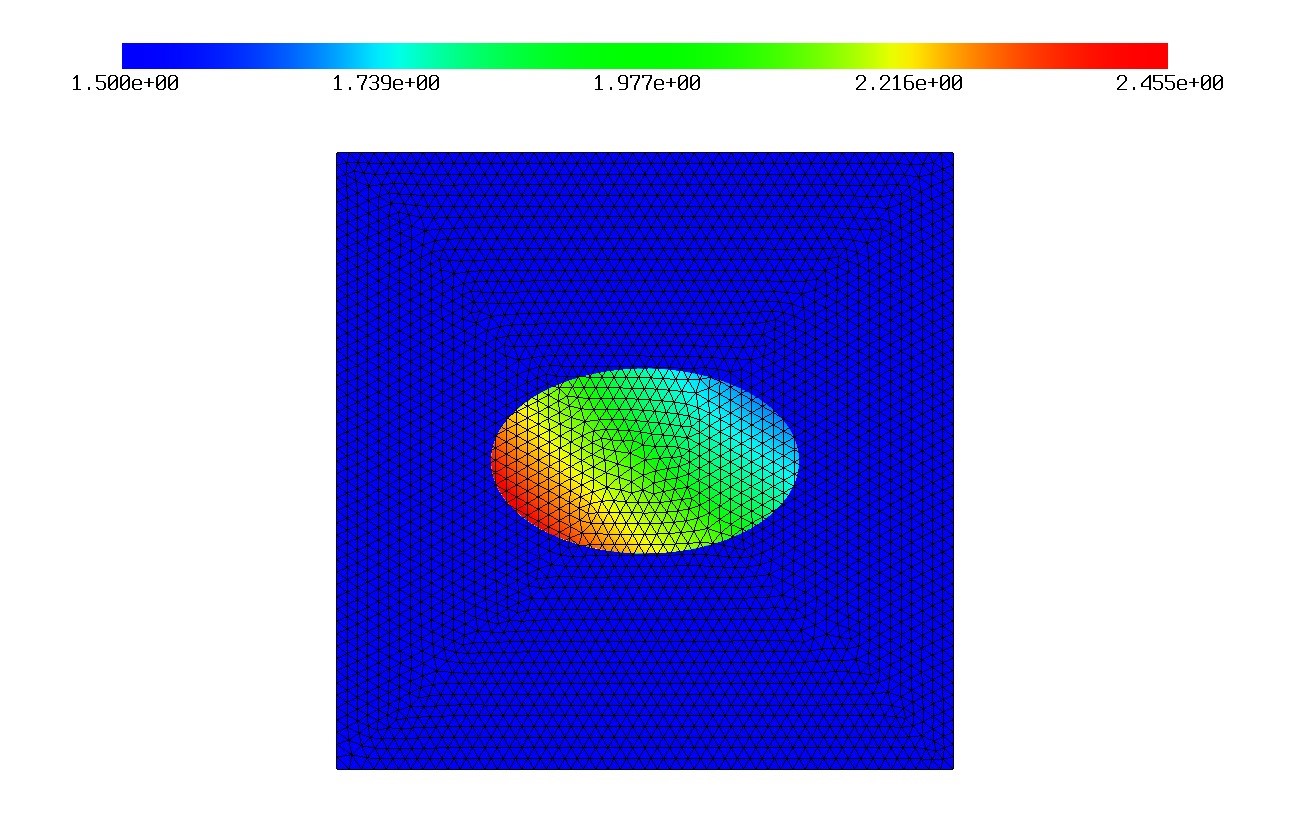}}
  \subfigure[reconstructed $\omega$ without noise]{
  \includegraphics[width=2.5in]{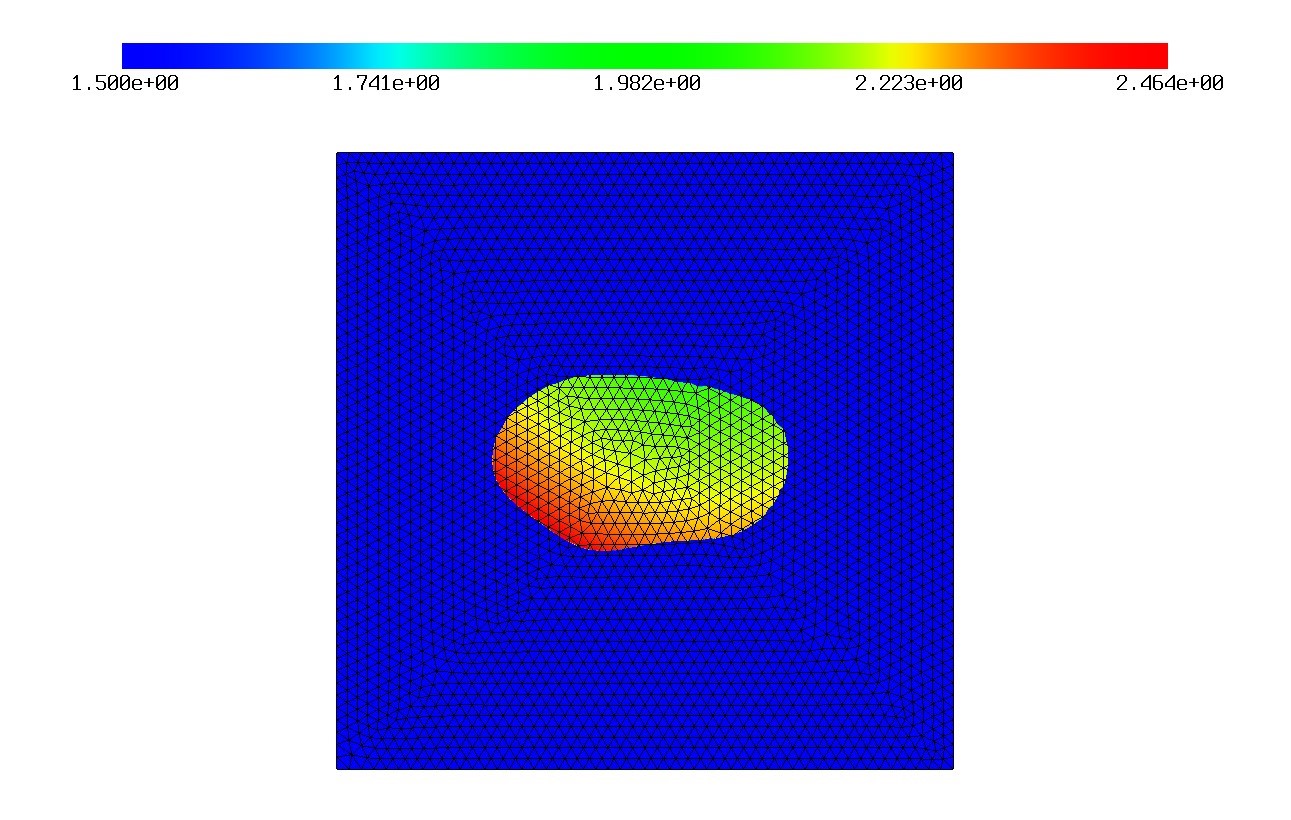}}\hspace{10mm}
  \subfigure[reconstructed $\omega$ with  $10\%$  noise]{
  \includegraphics[width=2.5in]{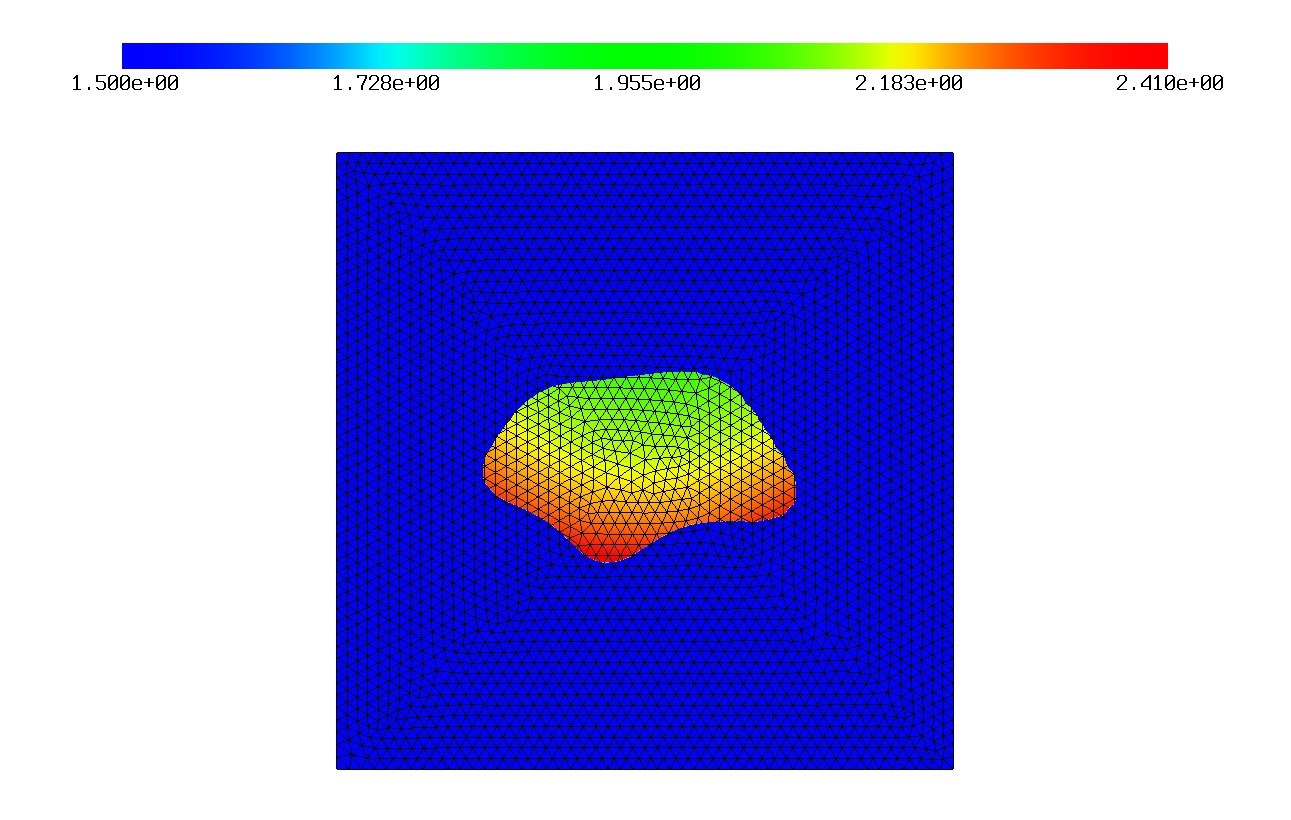}}
  \caption{Numerical results for Example \ref{e3} with different noise levels.}
  \label{fig32}
\end{figure}

In this example, we observe that the reconstructed results are sensitive to the Dirichlet boundary value. Any changes in support can lead to a significant variation in intensity. For the noise-free case, we obtain $\mbox{err}(q) = 1.50 \times 10^{-1}$.  However, when considering noise, the shape deviates significantly from the true position.
\end{Example}

\begin{Example}\label{e4}
In the fourth example we set $u_n = \sin(\pi x)\sin(\pi y)$ on $\Gamma$, $f = 1$ in $\Omega$, and the exact source function $q_e = 2x(1-x)+2y(1-y)$. The exact domain $\omega_e$ is defined as $\{(x,y) \in \Omega: x^2 + y^2 < 0.09 \}$. 

We choose the initial domain as $\omega_0 = \{(x,y) \in \Omega: (x-0.3)^2 + (y-0.3)^2 < 0.0225 \}$. The intensity error $\mbox{err}(q)$ reaches $5.31 \times 10^{-2}$, while the reconstructed support and strength are shown as in Figure \ref{fig4}, indicating a satisfactory result.

\begin{figure}[!htbp]
  \centering
  \subfigure[initial domain  $\omega_0$]{
  \includegraphics[width=2.5in]{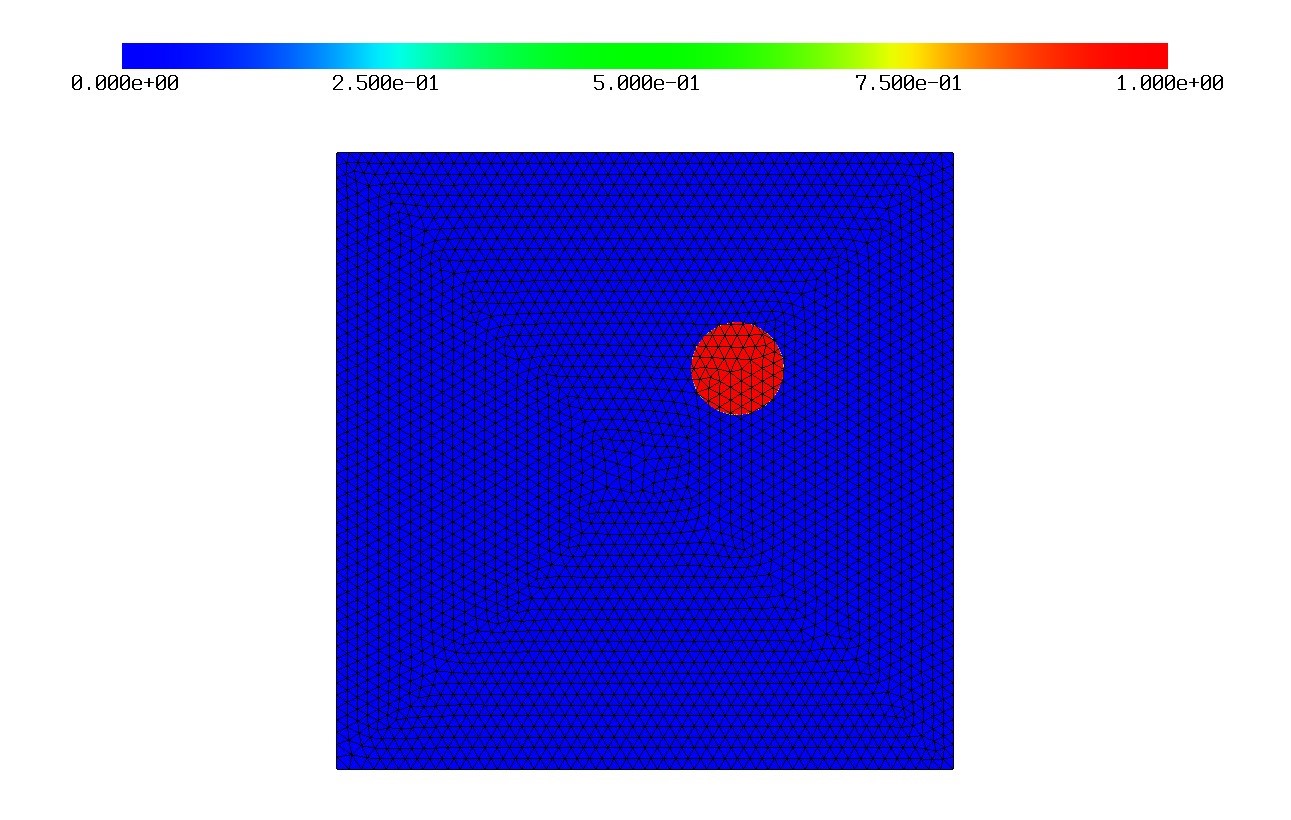}}\hspace{10mm}
  \subfigure[true source value]{
  \includegraphics[width=2.5in]{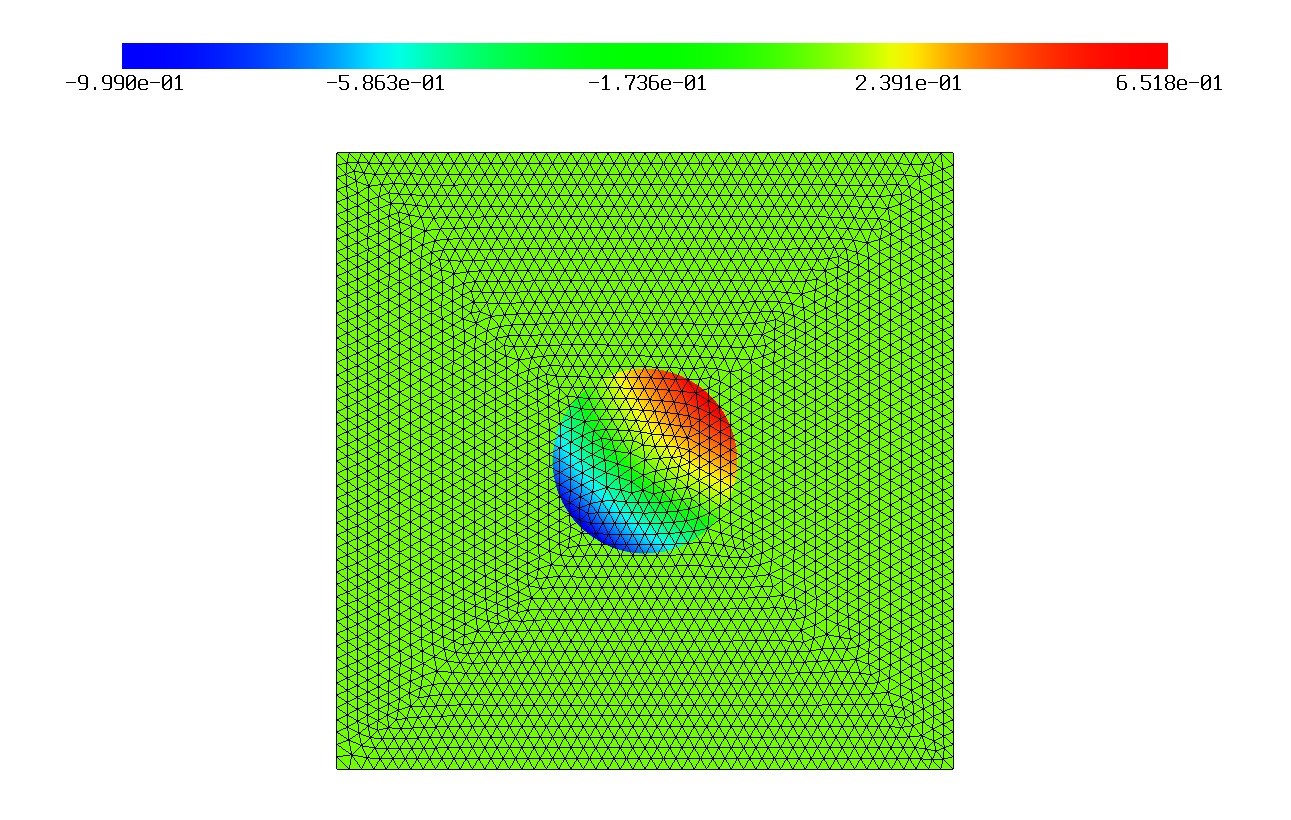}}
  \subfigure[reconstructed $\omega$ without noise]{
  \includegraphics[width=2.5in]{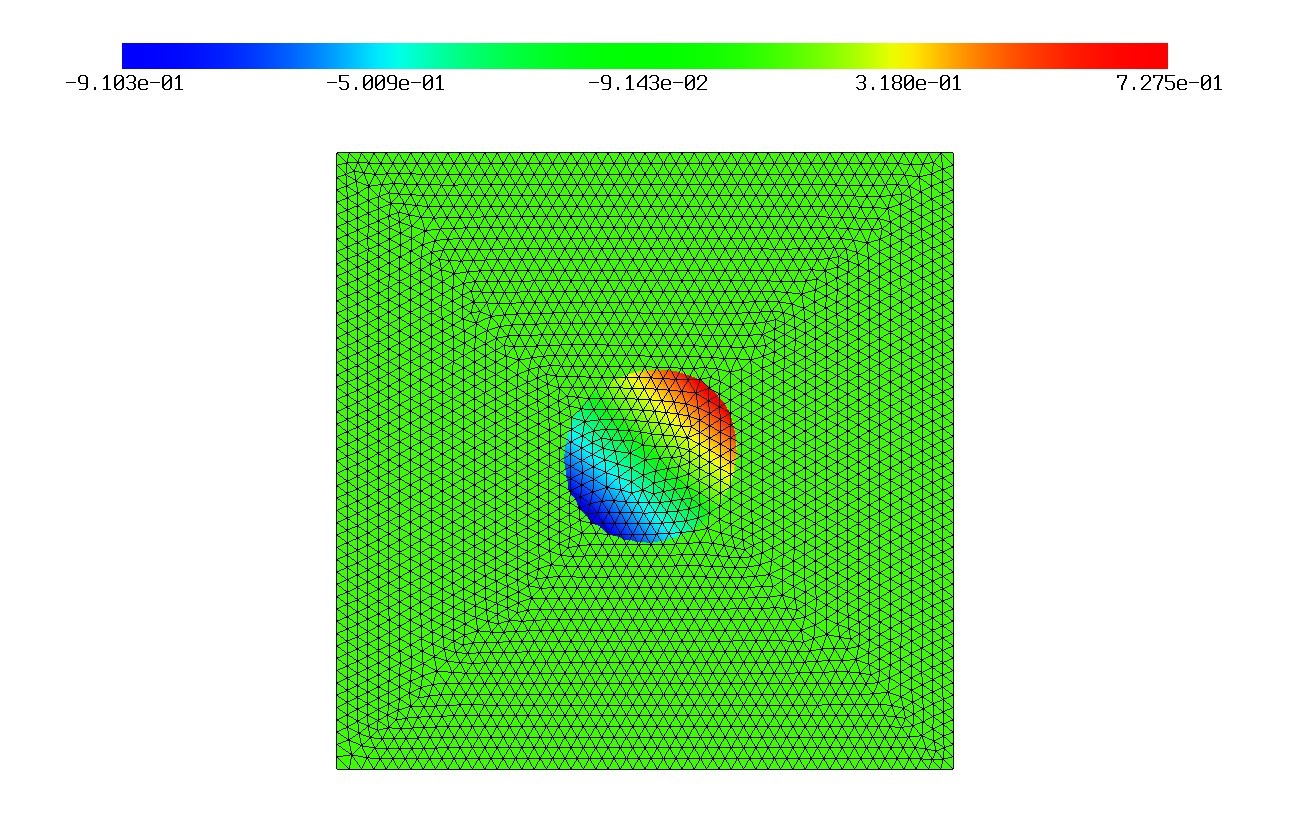}}\hspace{10mm}
  \subfigure[reconstructed $\omega$ with $10\%$ noise]{
  \includegraphics[width=2.5in]{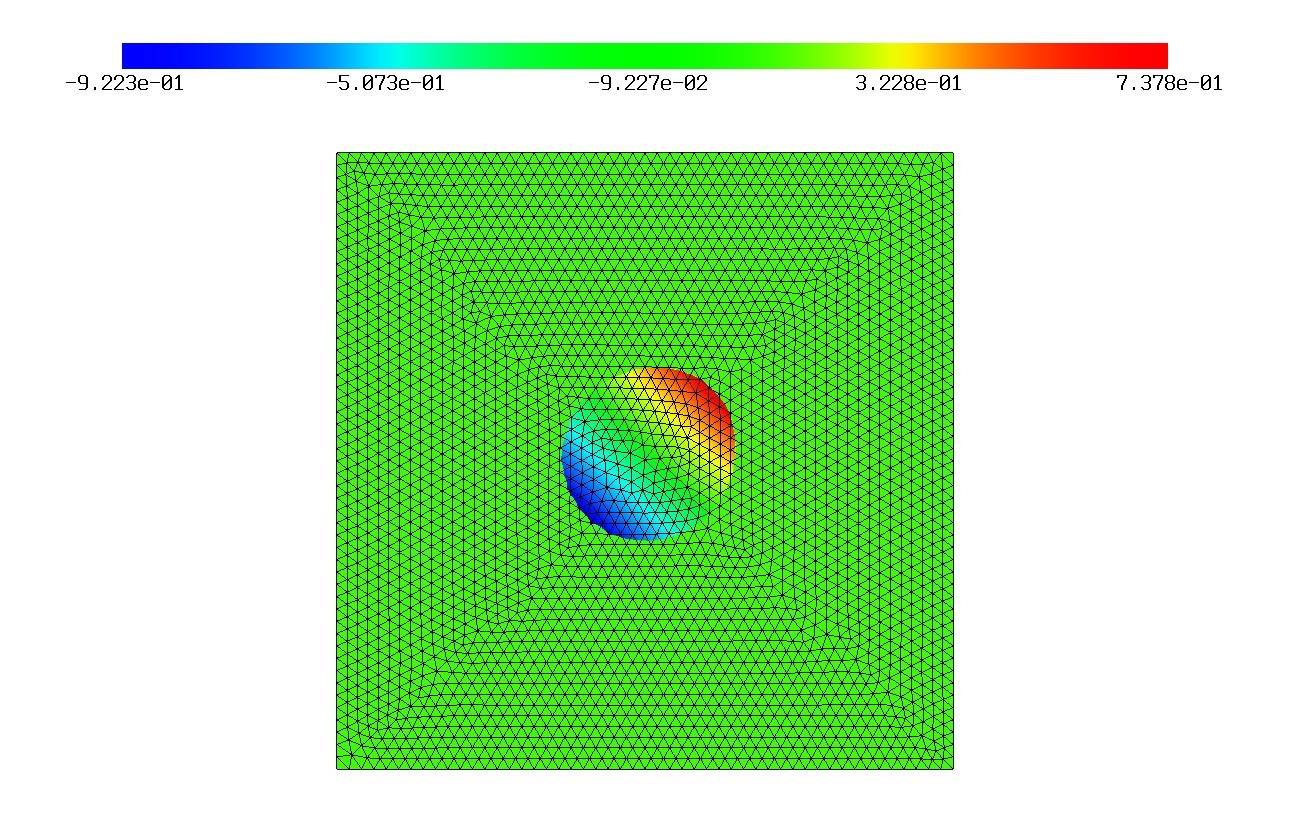}}
  \caption{Numerical results for Example \ref{e4} with different noise levels.}
  \label{fig4}
\end{figure}
\end{Example}

\begin{Example}\label{e5}
In this example, we aim to demonstrate the ability of the proposed algorithm to handle topological changes. We set $u_n = \sin(\pi x)\sin(\pi y)$ on $\Gamma$, $f = 1$ in $\Omega$, and define the exact source function $q_e = 1$ as well as the exact domain $\omega_e = \{(x,y) \in \Omega: (x\pm 0.45)^2 + (y\pm 0.45)^2 < 0.04 \}$. The data $u_d$ is computed in the same manner as above. 

\begin{figure}[!htbp]
  \centering
  \subfigure[initial domain $\omega_0$]{
  \includegraphics[width=2.5in]{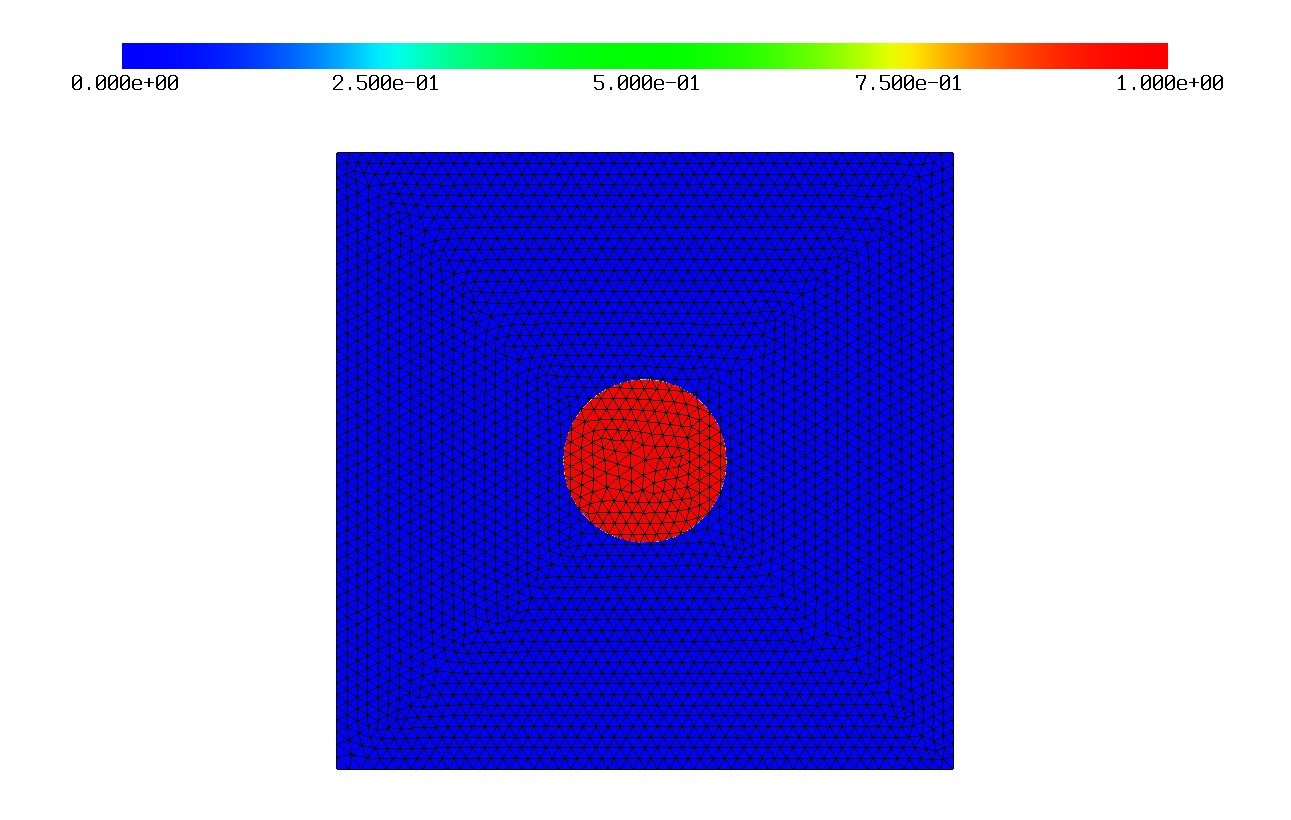}}\hspace{10mm}
  \subfigure[true source value]{
  \includegraphics[width=2.5in]{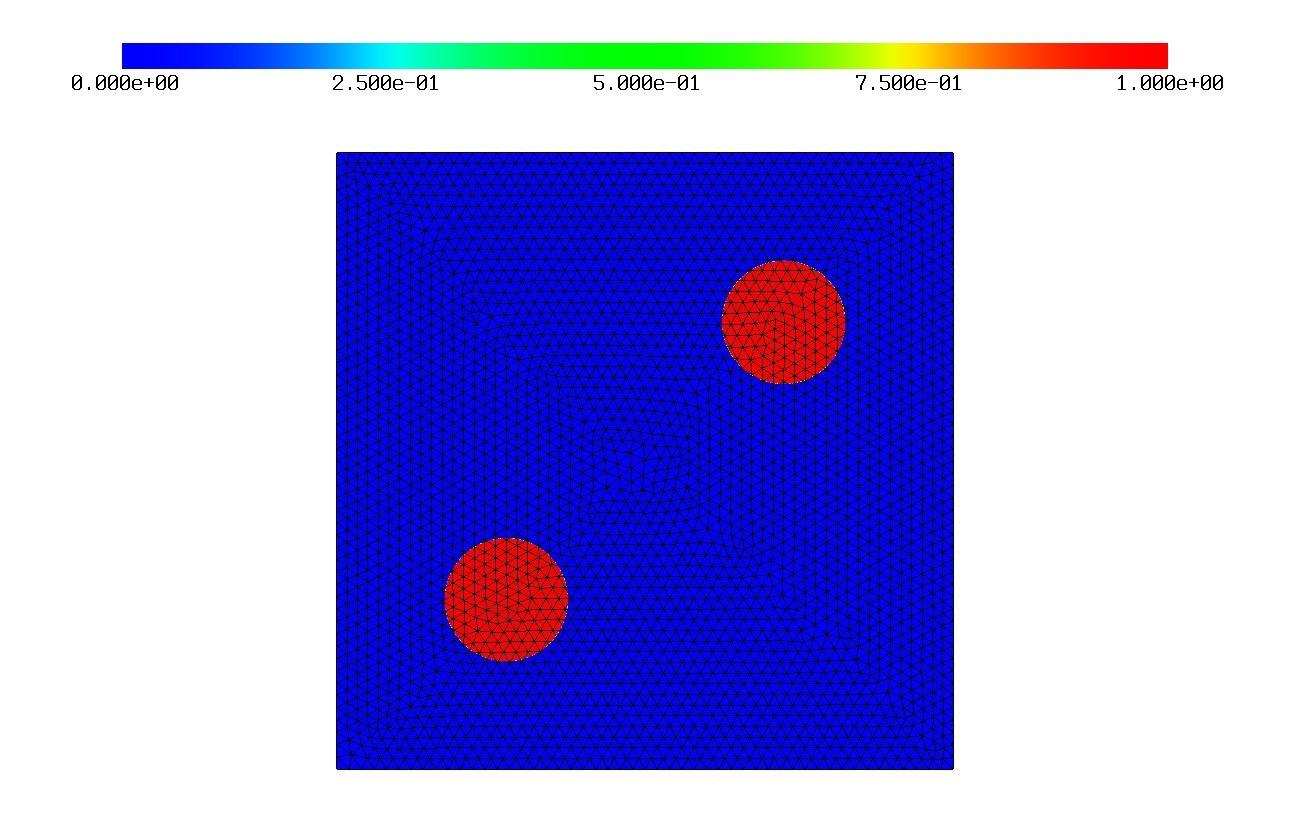}}
  \subfigure[reconstructed $\omega$ without noise]{
  \includegraphics[width=2.5in]{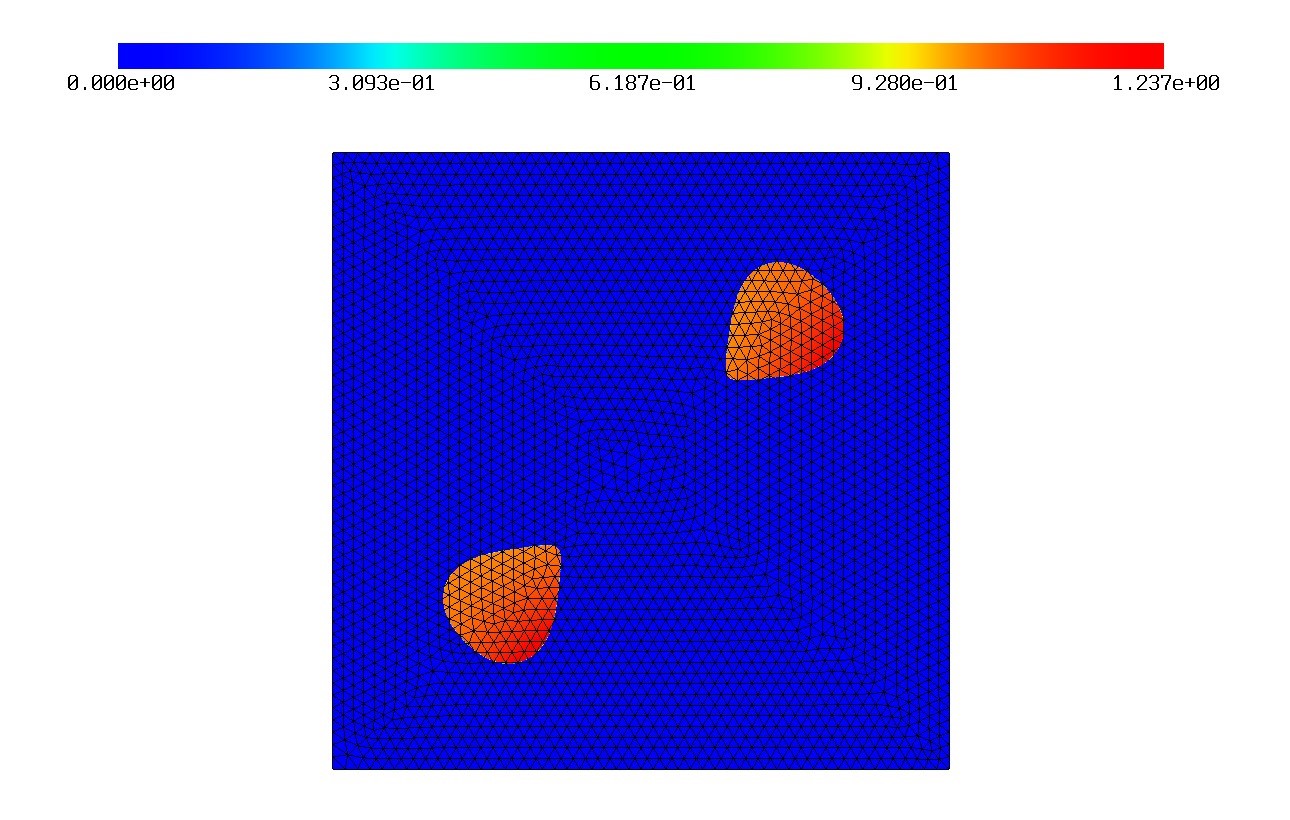}}\hspace{10mm}
  \subfigure[reconstructed $\omega$ with  $10\%$  noise]{
  \includegraphics[width=2.5in]{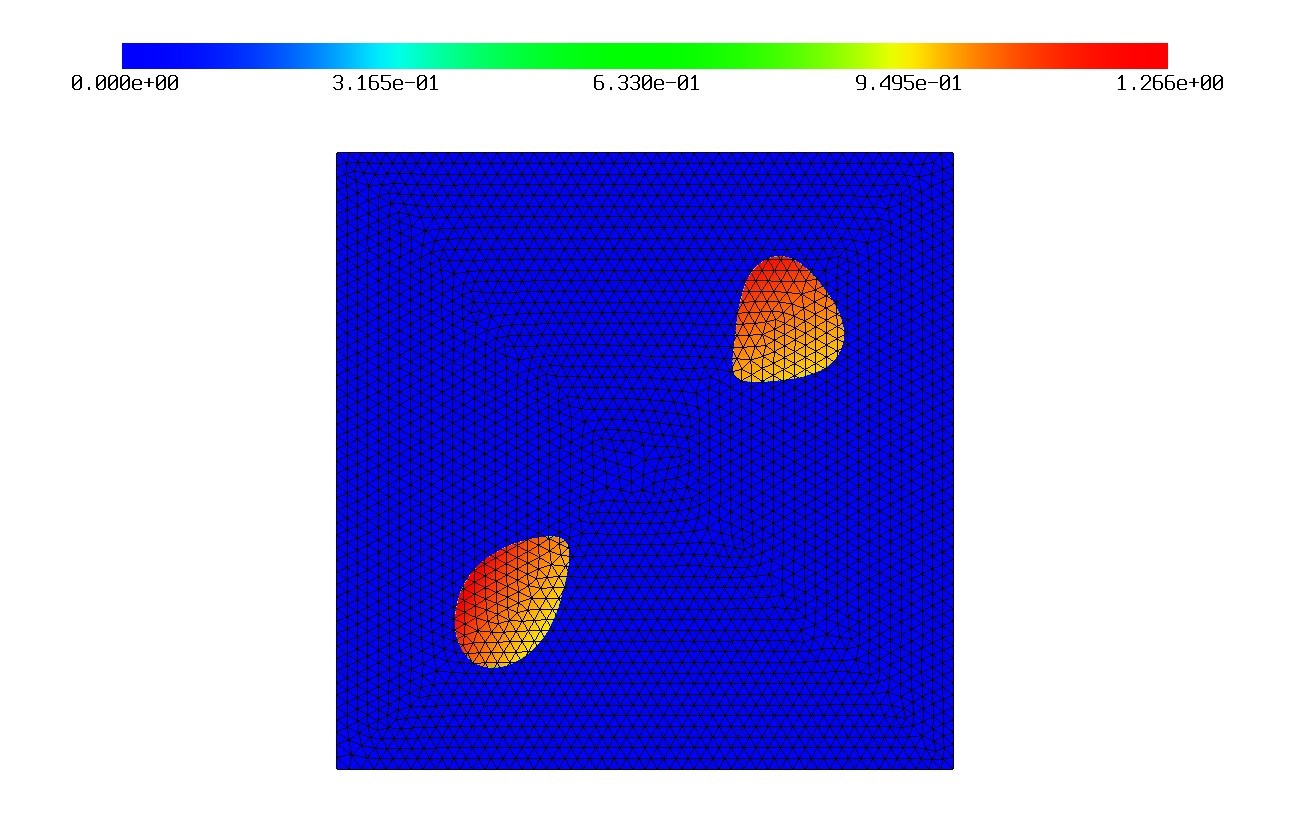}}
  \caption{Numerical results for Example \ref{e5} with different noise levels.}
  \label{fig5}
\end{figure}

We choose the initial domain as $\omega_0 = \{(x,y) \in \Omega: x^2 + y^2 < 0.07 \}$ with the same mesh size and finite element space as above. The intensity  error $\mbox{err}(q)$ reaches $4.26 \times 10^{-2}$. 
\end{Example}

\begin{Example}\label{e6}
In the final example, we demonstrate the proposed algorithm's ability to merge multi-connected domains, in contrast to the previous example where the domain was split. We set $u_n = \sin(\pi x)\sin(\pi y)$ on $\Gamma$, $f = 1$ in $\Omega$, with the exact source function $q_e = 2$ and the exact domain $\omega_e = \{(x,y) \in \Omega: x^2 + y^2 < 0.15 \}$. The data $u_d$ is computed in the same manner as above. 

\begin{figure}[!htbp]
  \centering
  \subfigure[initial domain $\omega_0$]{
  \includegraphics[width=2.5in]{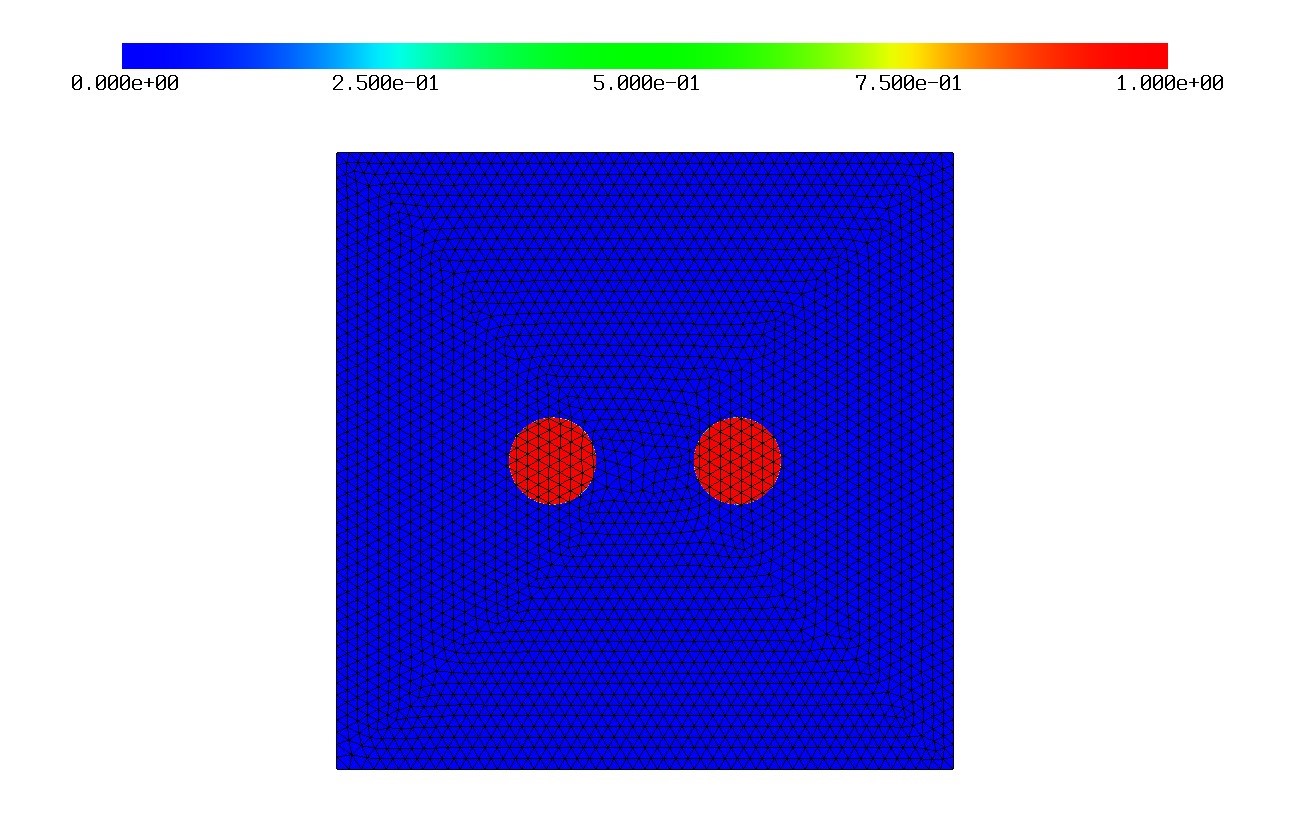}}\hspace{10mm}
  \subfigure[true source value]{
  \includegraphics[width=2.5in]{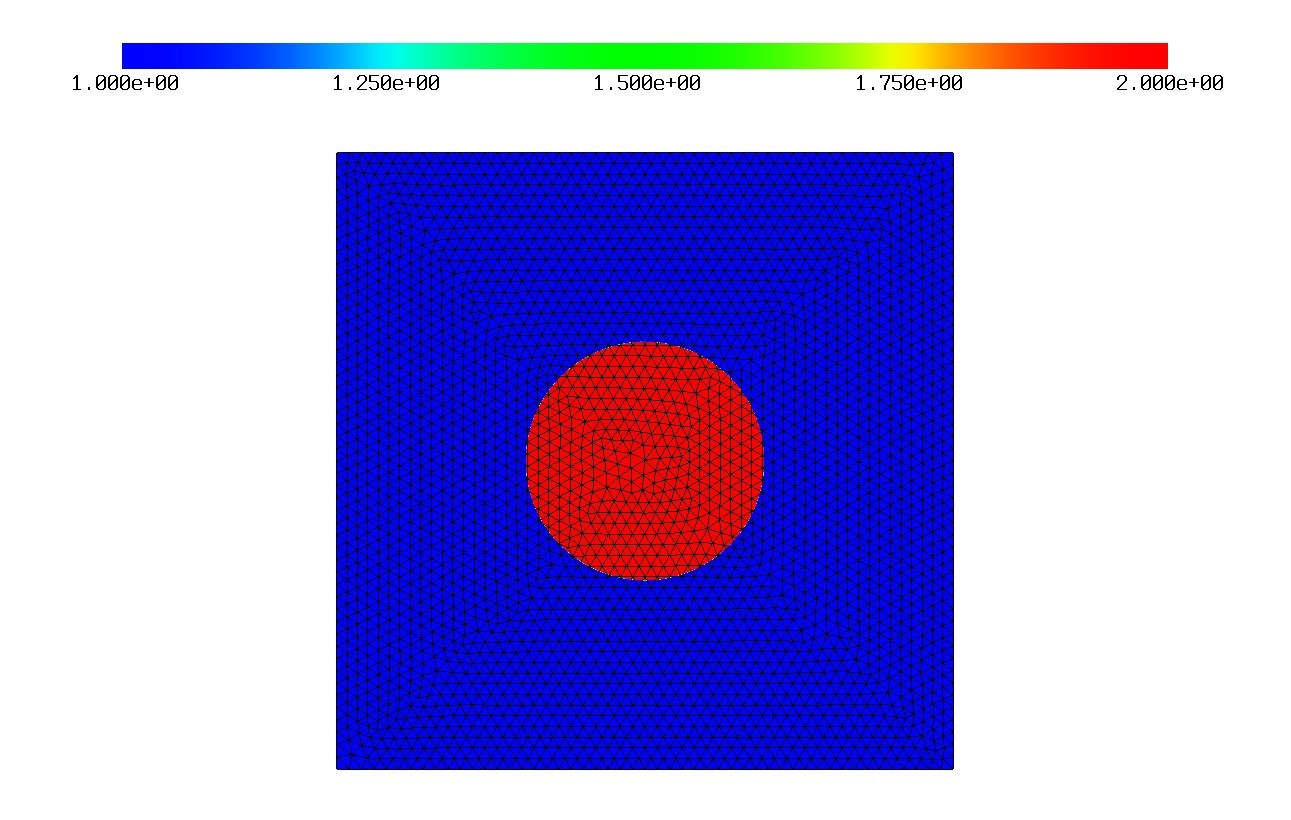}}
  \subfigure[reconstructed $\omega$ without noise]{
  \includegraphics[width=2.5in]{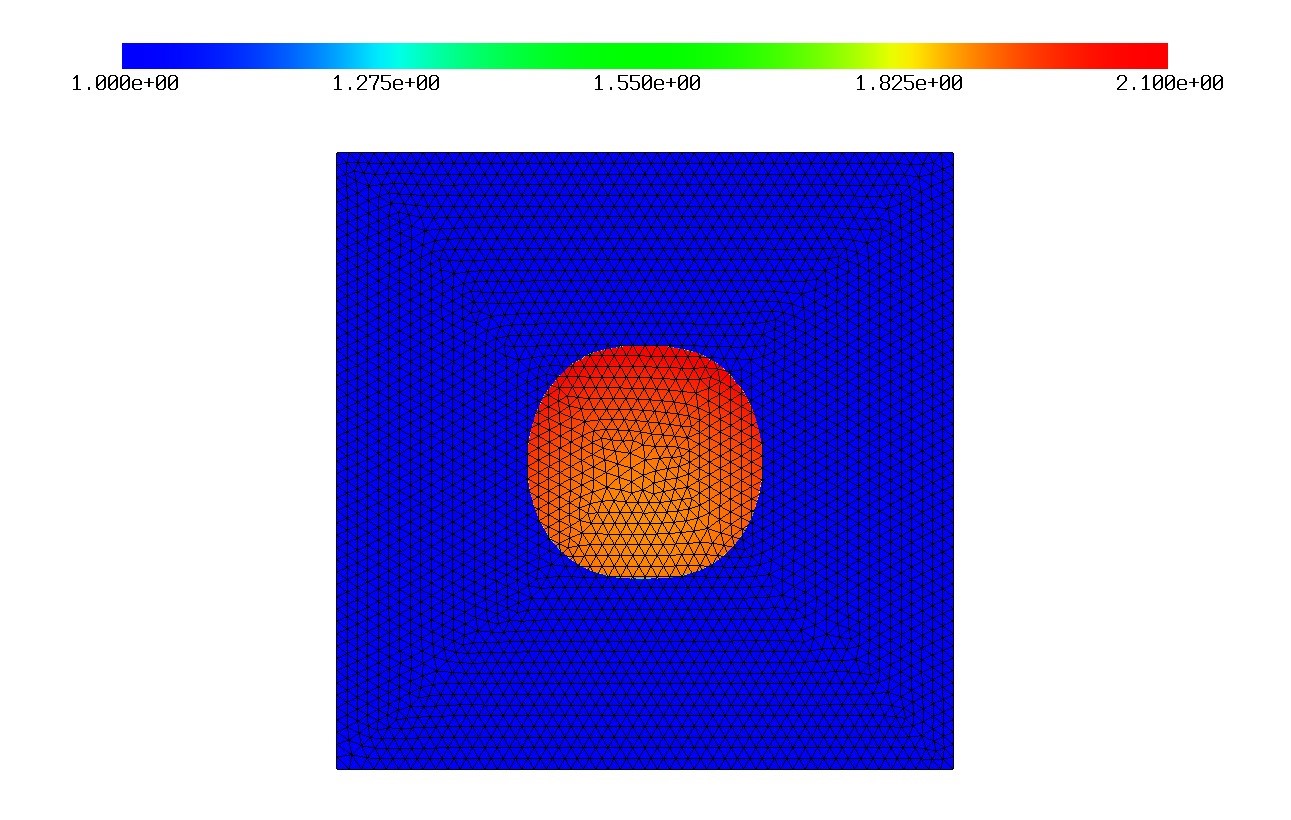}}\hspace{10mm}
  \subfigure[reconstructed $\omega$ with  $10\%$  noise]{
  \includegraphics[width=2.5in]{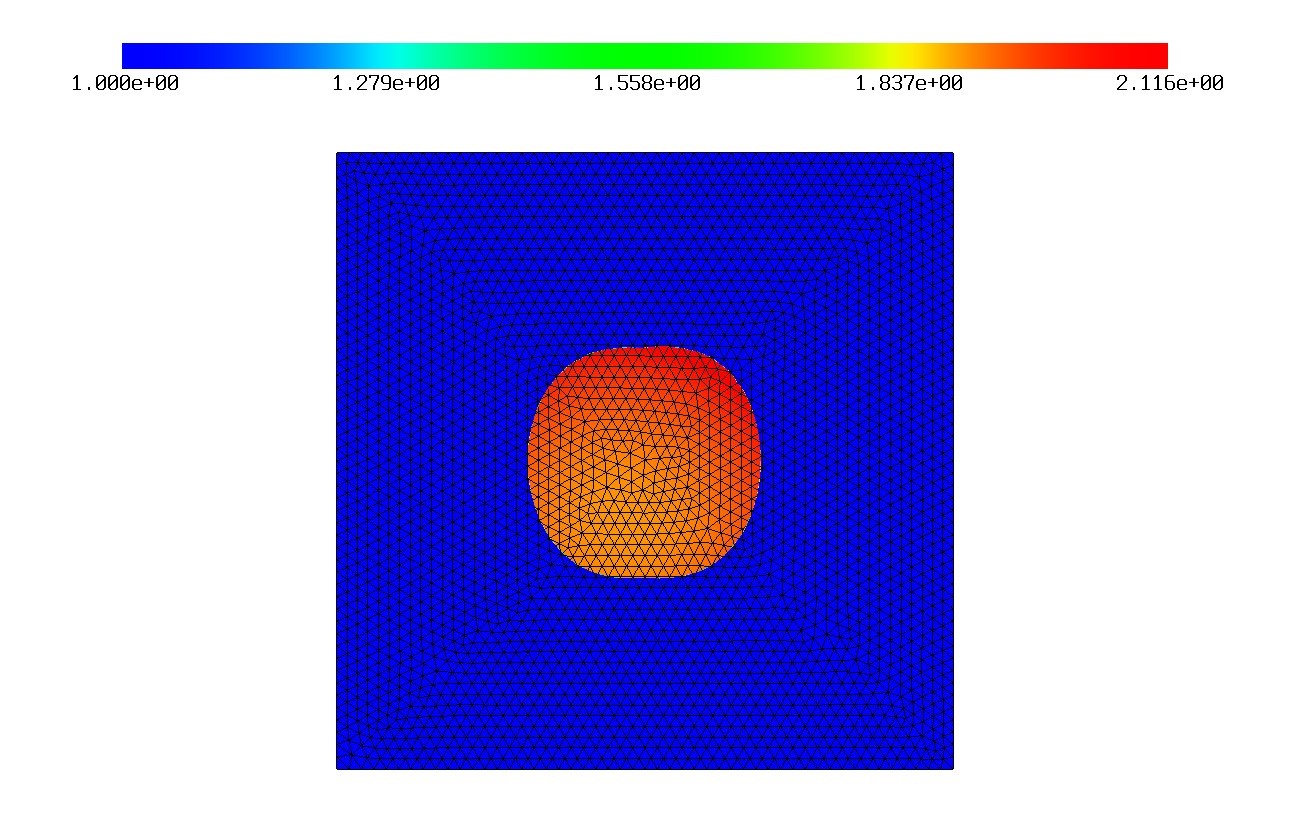}}
  \caption{Numerical results for Example \ref{e6} with different noise levels.}
  \label{fig6}
\end{figure}

We choose the initial domain as $\omega_0 = \{(x,y) \in \Omega: (x\pm 0.3)^2 + y^2 < 0.02 \}$ with the same mesh size and finite element space as above. The intensity  error $\mbox{err}(q)$ reaches $3.11 \times 10^{-2}$. When topology changes occur, the perimeter constraint becomes much more important and needs to be chosen carefully. Figure \ref{fig5} together with Figure \ref{fig6} illustrates that the algorithm can handle topology changes effectively for the source inversion.
\end{Example}

Based on the extensive numerical experiments conducted, we can draw several conclusions. Firstly, the proposed algorithm demonstrates satisfactory stability and accuracy, with noise having only a minimal impact on the results. Secondly, the algorithm is likely to fail if the initial domain is significantly far from the exact source support, consistent with the well-known sensitivity to the initial guess due to the non-convex and ill-posed nature of this type of inverse problem. Lastly, while the method performs well for constant sources, it does not always converge to an acceptable result for non-constant cases. When the source intensity varies significantly, treating the regularization parameter $\alpha$ as a variable may improve outcomes.

\section{Conclusion and perspective}
In this paper, we present a general shape optimization approach for identifying both the source strength and support. The crucial element of this approach is the Tikhonov regularization term in the least-squares formulation of the inverse problem, which enables the elimination of the source strength from the first-order optimality condition. We believe that the proposed approach can be extended in several directions.

First, we expect the approach to be adaptable to unsteady diffusion or subdiffusion equations, such as the heat equation or fractional diffusion equations. Second, we can incorporate pointwise constraints on the strength, e.g., $a\leq q(x)\leq b$ almost everywhere in $\omega$ for some constants $a$ and $b$. In this case, the first-order optimality condition (\ref{opt}) is replaced by
\begin{eqnarray}
J'_0(s-q)=(\alpha q+p,s-q)_\omega\geq 0\quad \forall s\in L^2( \omega), \ a\leq s(x)\leq b\ \mbox{a.e.\ in}\ \omega,\label{opt_constraint} 
\end{eqnarray}
which leads to the relation $q=\max\{a,\min\{-\frac{1}{\alpha}p,b\}\}$ in $\omega$. The state system of the resulting shape optimization problem (\ref{shape_obj})-(\ref{shape_state}) is not differentiable due to the non-smooth structure of $q$. However, this can be managed by employing smoothing techniques and then passing to the limit (cf. \cite{LuftSchulzWelker}).  Lastly, we anticipate extending the approach to recover potentials described by the equation 
\begin{eqnarray}
-\Delta u+\chi_\omega qu=f\quad\mbox{in}\ \Omega,
\end{eqnarray}
where both $\omega$ and $q$ are unknown. This type of problem has been extensively studied in the literature, typically under the assumption that the support $\omega$ is either known a priori or equal to the entire spatial domain (cf. \cite{JinLuQuanZhou}), or that the strength $q$ is known while the goal is to recover the support $\omega$ (cf. \cite{HettlichRundell1997}).


 \medskip

\end{document}